\documentclass[10pt]{article}
\usepackage{hyperref}
\hypersetup{hypertex=red,
	colorlinks=true,
	linkcolor=red,
	anchorcolor=red,
	citecolor=red}

\usepackage{psfrag}
\usepackage{epstopdf}
\usepackage{epsfig}
\usepackage{verbatim}
\usepackage{fancyhdr}
\usepackage{subfigure}
\usepackage{indentfirst}
\usepackage{booktabs}
\usepackage{color}
\usepackage{mathrsfs}
\usepackage{graphics}
\usepackage{graphicx}
\usepackage{listings}
\usepackage{xcolor}
\usepackage{float}
\usepackage{multirow}
 
\usepackage{amsmath}
\usepackage{amssymb}
\usepackage{amscd}
\usepackage{amsthm}

\usepackage[all]{xy}

\usepackage{algorithm}  
\usepackage{algorithmic}  
\usepackage{enumerate}

\usepackage{xypic}  

\usepackage{My_macro_unified}
\usepackage{My_macro_special}

\begin{document}
	
\title{\textbf{Auxiliary iterative schemes for the discrete operators on de Rham complex}}

\author{ Zhongjie Lu 
	\thanks{School of Mathematical Sciences, 
		University of Science and Technology of China, Hefei, Anhui 230026, China.
	Email: zhjlu@ustc.edu.cn. Research supported by NSFC grant No. 12101586.}}

\date{}	
%\date{\color{red}\today}
	
	\maketitle
	
	\begin{abstract}
					
	The main difficulty in solving the discrete source 
	and eigenvalue problems of the operator $ d^*d $
	with iterative methods
	is to deal with the huge kernels,
	for example,	
	the $ \nabla \times \nabla \times $ 
	and $- \nabla  \left( \nabla \cdot\right) $ operator.
	In this paper, 
	we construct auxiliary schemes
	for their discrete problems
	based on Hodge Laplacian on de Rham complex.
	The spectra of the auxiliary schemes are Laplace-like.
	Then many efficient iterative methods 
	and preconditioning techniques
	can be applied to them.
	After getting the solutions of the auxiliary schemes,
	the desired solutions of the original problems can be
	recovered or identified through some simple operations.
%	We sum these up as a new framework to compute 
%	the discrete source and eigenvalue
%	problems of the operator $ d^*d $ using iterative method.
	We also investigate two preconditioners for the auxiliary schemes,
	ILU-type method and Multigrid method.
	Finally,
	we present some numerical experiments
	to verify the efficiency of the auxiliary schemes.

	\textbf{Keywords: de Rham complex, Maxwell equation, iterative method, multigrid method }

	\end{abstract}

   %\clearpage
   
   \section{Introduction}
   
   In this paper,
   we consider the algebraic systems involving the operator $ d^*d $.
   Here, $ d $ is the differential operator on de Rham complex
   \begin{equation}\label{deRham_complex}
   	\begin{split}
   		\xymatrix{
   			0 \ar[r] & V^0 \ar[r]^{d^0} 
   			& V^1 \ar[r]^{d^1} 
   			& \cdots \ar[r]^{d^{n-1}} 
   			& V^n \ar[r]& 0
   		}
   	\end{split}
   \end{equation}
   and $ d^* $ is its adjoint operator.
   The main difficulty in solving such problems with iterative methods
   is to deal with the kernel of $ d^*d $.
   We take the $ \mbR^3 $ complex 
   \begin{equation}\label{R3_complex}
   	\begin{split}
   		\xymatrix{
   			0 \ar[r] & H(\textgrad) \ar[r]^{\nabla} 
   			& \vH(\textcurl) \ar[r]^{\nabla \times} 
   			& \vH(\textdive) \ar[r]^{\nabla \cdot} 
   			& L^2 \ar[r]& 0
   		}
   	\end{split}
   \end{equation}
   as an example
   to illustrate the problems that we want to deal with.

   The $ 0 $-form $ d^*d $ operator on the complex \eqref{R3_complex}
   is the scalar Laplace operator $ -\Delta $,
   the $ 1 $-form is the Maxwell operator $ \nabla \times \nabla \times $
   and the $ 2 $-form is the grad-div operator $ -\nabla(\nabla \cdot ) $.
   When a new iterative method or preconditioning technique is proposed,
   the Laplace operator is always a standard test model.
   The high efficiency in solving the discrete Laplace problems
   is almost the basic requirement for
   a good iterative technique.
   There are plenty of iterative techniques satisfying such requirement 
   \cite{greenbaum1997iterative, MR1990645}.
   The main difference between the scalar Laplace operator
   and the following two operators is 
   the dimensions of their kernels.
   The kernel of the scalar Laplace operator 
   is usually caused by the shape of the domains
   or its boundary conditions,
   and its dimension is a limited number.
   From the complex \eqref{R3_complex},
   we can find that
   the kernel of the Maxwell operator contains 
   the space $ \nabla H(\textgrad) $
   and the kernel of the grad-div operator contains
   the space $\nabla \times H(\textcurl) $.
   They are both infinite dimensional.
   When discretizing the two operators,
   there exist huge kernels in their algebraic systems,
   especially for large-scale problems.
   These kernels cause essential difficulties 
   in solving large algebraic systems
   with iterative methods.
   
   One of the two main features of the multigrid method
   is error correction on a coarse grid \cite{MR2373954}.
   We discretize the Maxwell equation 
   \begin{equation}\label{Maxwell}
   	\begin{split}
   		\nabla \times \nabla \times \vu + c\vu = \vf
   	\end{split}
   \end{equation}
   using the first order edge finite element
   in a cubic domain $ [0,1]^3 $,
   and use a  $ 100\times 100 \times 100 $ uniformly cubic mesh.
   Then the total degree of freedoms of the discrete problem of \eqref{Maxwell}
   is the number of the edges in this mesh,
   which is roughly $ 3\times 100^3 = 3,000,000 $.
   The gradient of the nodal element space is contained 
   in the kernel of the Maxwell operator $ \nabla\times \nabla \times $,
   the dimension of which is about $ 100^3 = 1,000,000 $.
   The frequencies of the modes in the $ 1,000,000 $-dimensional kernel are all $ c $
   in the discrete system of \eqref{Maxwell}.
   A nature coarse mesh is $ 50\times 50 \times 50 $.
   The degree of freedoms of this 'coarse' problem
   is about $ 3\times 50^3 = 375,000 $.
   This is even small than the dimension 
   of the kernel (the low-frequency modes) in the 'fine' problem.
   We could not expect that 
   this 'coarse' problem can correct the errors
   of the low-frequency modes with much larger dimension.
   The other feature of the multigrid method
   is smoothing on the current grid.
   In fact, many smoothers that are efficient for the Laplace-like problems
   do not work for the discrete Maxwell problems.
   
   In the Maxwell eigenvalue problem
   \begin{equation}\label{Maxwell_eig}
   	\begin{split}
   		\nabla \times \nabla \times \vu  = \lambda \vu,
   	\end{split}
   \end{equation}
   the meaningful eigenpairs are usually two parts in practical applications.
   The first part are nonzero eigenvalues, 
   especially some of the smallest nonzero eigenvalues.
   The smallest eigenvalue of the discrete system 
   of the Maxwell eigenproblem \eqref{Maxwell_eig}
   is $ 0 $ with the same dimension of the huge discrete kernel.
   However, many iterative algorithms for matrix eigenvalue problems
   are efficient in finding the smallest or largest  
   eigenvalues of a system \cite{template_eigen}.
   The large number of the zero eigenvalues 
   are the obstacle in front of the smallest nonzero eigenvalues.
   The second part of the meaningful eigenpairs 
   are the solenoidal and irrotational (harmonic) functions,
   whose corresponding eigenvalues are also $ 0 $.
   Then how to separate them from the large number of the eigenpairs 
   with $ 0 $ eigenvalue is a problem.
   Spectral inverse and shift method
   or adding a Lagrange multiple
   can be a choice in dealing with this problem.
   The two methods involve solving linear systems.
   If the size of the discrete problem is not very large,
   the system can be solved by  direct  solver.
   When the size is large,
   the systems have to be solved by iterative method.
   Then it comes back to the problems in using
   iterative solvers for the systems with huge kernels.

   The efficient algorithms in solving the discrete systems 
   involving the Maxwell operator and grad-div operator
   have been studied by many researchers.
   It is still an on-going topic.
   \cite{hiptmair1998multigrid} proposed a multigrid method for Maxwell equation
   by treating the kernel and its orthogonal complement separately.
   \cite{arnold2000multigrid} constructed proper Schwarz smoothers for 
   multigrid in $ H(\textdive) $ and $ H(\textcurl) $. 
   \cite{hiptmair2007nodal} proposed nodal auxiliary space preconditioners
   for the two spaces based on regular decompositions.
   \cite{hiptmair2002multilevel} constructed multilevel method for
   Maxwell and grad-div eigenvalue problems.
   To the author's knowledge,
   the review of the researches on this problem
   is far from complete.

   The approximation for $ d^*d $ problems is also a very important topic.
   The discrete operator considered in this paper is base on 
   these studies of approximation.
   We only cite some general results of them here.
   For details, 
   we refer the readers to 
   \cite{arnold2006finite, 
   	boffi2010finite, 
   	boffi2013mixed,
   	boffi2000problem,
   	boffi1997convergence,
   	MR2009375,
   	Kikuchi1987,
   	Kikuchi1989} 
   and the references therein.

   In this paper,
   we construct   auxiliary schemes
   for the  discrete $ d^*d $ problems using Hodge Laplacian.
   The auxiliary part complements the kernel of the discrete $ d^*d $ operators.
   Then the spectrum of the auxiliary schemes become Laplace-like.
   The auxiliary schemes for source and eigenvalue problems 
   can be both computed in the same way as Laplace problems.
   Many efficient iterative methods and preconditioning techniques
   that are efficient for Laplace problems
   can be also applied to these auxiliary schemes 
   with simple modifications.
   After obtaining the solutions of the auxiliary schemes,
   the desired solutions can be recovered or identified
   through simple operations.
   We also consider an ILU-type method and multigrid method
   to solve the auxiliary schemes.
   Finial,
   we take the Maxwell and grad-div source and eigenvalue problems
   to verify the auxiliary schemes.
   
   The auxiliary schemes proposed in this paper contain two steps:
   \begin{itemize}
   	\item Add an auxiliary term to the original problems
   	and solve the auxiliary problems.
   	\item Recover or identify the desired solutions 
   	of the original problems 
   	from the solutions of the auxiliary problems.
   \end{itemize}
   In the first step, 
   we derive the auxiliary problems 
   from the discrete Hodge Laplacian 
   in finite element exterior calculus (FEEC).
   In fact,
   the auxiliary problems have the similar matrix form
   as the penalty methods,
   the discrete regularization
   and the grad-div stabilisation methods.
 %  in \cite{MR2009375} as the referee mentioned.
   The crucial point is the second step.
   Since the matrix forms contain an extra term,
   their solutions are no longer the solutions 
   of the original problems.
   As the auxiliary term is constructed 
   according to FEEC,
   it contains the similar property
   to $ d^{k+1}\circ d^k = 0 $ on complex.
   Using this property on matrix level,
   the desired solutions 
   of the original problems 
   can be easily recovered or identified 
   from the solutions of the auxiliary problems.
   Beside the Laplace-like problem,
   it is enough (at most) to solve
   a mass equation,
   which usually has quite good condition
   and is easily to be solved using iterative methods.
   As the auxiliary schemes are constructed in an abstract setting,
   except for the problems related to the operators
   $ \nabla \times \nabla \times $ and $- \nabla  \left( \nabla \cdot\right) $
   on $ \mbR^3 $-complex,
   other similar problems on de Rham complex
   can be solved using the same schemes.

   This paper is organized as follows.
   In Section \ref{sec_abstract}, we study the approximation of the abstract Hodge Laplacian on de Rham complex.
   In Section \ref{sec_analysis}, we deduce some properties of discrete operators.
   In Section \ref{sec_auxiliarye}, we use the discrete Hodge Laplacian to construct auxiliary schemes 
   for the discrete source and eigenvalue problems of $ d^*d $ operator.
   In Section \ref{matrix_aux}, we study the corresponding the matrix form of the auxiliary schemes 
   and simplify them.
   In Section \ref{sec_precond}, we investigate two iterative methods, ILU-type preconditioning and
   multigrid method, to solve the auxiliary schemes.
   In Section \ref{sec_numerical}, we take the source and eigenvalue of 
   $ \nabla \times \nabla \times $ and $- \nabla  \left( \nabla \cdot\right) $ operators
   as examples to verify the efficiency of the auxiliary schemes. 
   In Section \ref{sec_conclusion}, there are some conclusions.

   %\clearpage

\section{The abstract Hodge Laplacian and its approximation}\label{sec_abstract}

The $ d^*d $ operator can be viewed 
as a half of the Hodge Laplacian on de Rham complex.
In this section,
we study the 'full' Hodge Laplacian problems first.
The analysis is in the framework of the Hilbert complex.
The de Rham complex is a typical example of the Hilbert complex
when the operators are differential operators
and the spaces are the corresponding function spaces.
The theoretical framework in this section has been 
established in \cite{arnold2018finite, arnold2006finite, arnold2010finite}.
We refer the readers to these references 
for more details.

\subsection{The abstract Hodge Laplacian}
Let us consider a Hilbert complex $ (W,d) $.
The $ d^k $ is a closed densely defined operator
from $ W^k $ to $ W^{k+1} $ 
and its domain is denoted by $ V^k $.
The corresponding  domain complex is
\begin{equation}\label{Hcomplex}
	\begin{split}
		\xymatrix{
			0 \ar[r] & V^0 \ar[r]^{d^0} 
			& V^1 \ar[r]^{d^1} 
			& \cdots \ar[r]^{d^{n-1}} 
			& V^n \ar[r]& 0.
		}
	\end{split}
\end{equation}
The adjoint operator of $ d^k $ is denoted by
$ (d^k)^*: W^{k+1} \to W^{k} $
and is defined as
\begin{equation}\nonumber %label{}
	\begin{split}
		\left\langle  (d^k)^*\vu \, ,\, \vv \right\rangle 
		=\left\langle \vu \, ,\,   d^k\vv \right\rangle,
	\end{split}
\end{equation}
if $ \vu \in V^{k} $ or $ \vv \in W^{k+1} $
vanishes near the boundary.
Its domain is a dense subset of $ W^{k+1} $
and denoted by $ V^*_{k+1} $.
Then we have the dual complex
\begin{equation}\label{Hcomplex_dual}
	\begin{split}
		\xymatrix{
			0 \ar[r] & V^*_n \ar[r]^{(d^{n-1})^*} 
			& V^*_{n-1} \ar[r]^{(d^{n-2})^*} 
			& \cdots \ar[r]^{(d^{0})^*} 
			& V^*_0 \ar[r]& 0.
		}
	\end{split}
\end{equation}
The range and the null spaces of the differential operators
are denotes by
\begin{equation}\nonumber %label{}
	\begin{split}
		\mfB^k = d^{k-1}V^{k-1},
		\qq \mfZ^k = \mcN(d^k),
		\qq \mfB^*_k = (d^{k})^* V^*_{k+1},
		\qq \mfZ^*_k = \mcN((d^{k-1})^*).
	\end{split}
\end{equation}
The cohomology space is denoted by $ \mcH^k = \mfZ^k/\mfB^k $
and the space of harmonic $ k $-forms is denoted by
$ \mfH^k = \mfZ^k \cap \mfZ^*_k $.

In this paper,
we focus on the Hilbert complex with \emph{compactness property},
i.e. the inclusion $ V^k\cap V^*_k \subset W^k $ is compact for each $ k $.
In this case, 
the Hilbert complex is closed and Fredholm \cite[Theorem 4.4]{arnold2018finite}.
Then we have 
$$ \mcH^k \cong \mfH^k, $$
and their dimensions are finite.
There is the following Hodge decomposition \cite[Theorem 4.5]{arnold2018finite}
\begin{equation}\nonumber %label{}
	\begin{split}
		V^k = \mfB^k \oplus \mfH^k \oplus \mfZ^{k\bot_V}.
	\end{split}
\end{equation}
Here, $ \mfZ^{k\bot_V} = \mfB^*_k \cap V^k $.

The $ k $-form Hodge Laplacian operator is denoted by
\begin{equation}\label{Hodge_strong}
	\begin{split}
		L^k = (d^k)^* d^k  + d^{k-1} (d^{k-1})^*
	\end{split}
\end{equation}
and its domain is
\begin{equation}\nonumber %label{}
	\begin{split}
		D(L^k) = \left\lbrace \vu \in V^k \cap V_k^* \,|\, 
		d^k \vu \in V^*_{k+1} \; \text{ and } \; (d^{k-1})^* \vu \in V^{k-1} \right\rbrace.
	\end{split}
\end{equation}
For a positive number $ c>0 $, 
we consider the  problem:
given $ \vf \in W^k $, find $ \vu \in D(L^k) $ such that
\begin{equation}\label{Hodge_operator}
	\begin{split}
		(d^k)^* d^k \vu + d^{k-1} (d^{k-1})^* \vu + c \vu = \vf.
	\end{split}
\end{equation}
This problem can be written in a mixed weak formulation:
given $ \vf \in W^k $, 
find  $ (\vsigma,\vu)\in V^{k-1}\times V^{k}$
such that
\begin{equation}\label{Hodge_mixed}
	\begin{split}
		\left\langle \vsigma\, ,\, \vtau \right\rangle 
		- \left\langle \vu \, ,\,  d^{k-1}\vtau \right\rangle 
		&=0 \qqq\qq\; \vtau \in V^{k-1},\\
		\left\langle d^k \vu \, ,\,  d^k \vv \right\rangle 
		+ \left\langle d^{k-1}\vsigma \, ,\,  \vv \right\rangle 
		+ c\left\langle \vu \, ,\,  \vv \right\rangle 
		&= \left\langle \vf \, ,\,  \vv \right\rangle \qq \vv \in V^k.		%
	\end{split}
\end{equation}
We also consider the Hodge Laplacian eigenvalue problem:
find $ (\lambda,\vu) \in \mbR \times V^k $ such that
\begin{equation}\label{Hodge_eig_strong}
	\begin{split}
		(d^k)^* d^k \vu + d^{k-1} (d^{k-1})^* \vu  = \lambda \vu.
	\end{split}
\end{equation}
Its mixed formulation is:
find $ (\lambda,\vu) \in \mbR \times V^k $ such that
\begin{equation}\label{Hodge_eig}
	\begin{split}
		\left\langle \vsigma, \vtau \right\rangle 
		- \left\langle \vu \, ,\,  d^{k-1}\vtau \right\rangle 
		&=0 \qqq\qqq\; \vtau \in V^{k-1},\\
		\left\langle d^{k-1}\vsigma \, ,\,  \vv \right\rangle 
		+ \left\langle d^k \vu \, ,\,  d^k \vv \right\rangle
		&= \lambda \left\langle \vu \, ,\,  \vv \right\rangle \qq\, \vv \in V^k.
	\end{split}
\end{equation}
If the complex is Fredholm,
then there are at most a limit number of zero eigenvalues 
in \eqref{Hodge_eig_strong} and \eqref{Hodge_eig}.

\subsection{The approximation for Hodge Laplacian} \label{approx_Hodge}

The $ k $-form Hodge Laplacian
involves a segment of the complex \eqref{Hcomplex} with three spaces:
\begin{equation}\nonumber %label{}
	\begin{split}
		\xymatrix{
			V^{k-1} \ar[r]^{d^{k-1}} 
			&V^k \ar[r]^{d^k} 
			&V^{k+1}.
		}
	\end{split}
\end{equation}
We use the finite element spaces $ V^{k-1}_h$ and $ V^{k}_h$
to discretize the continuous problems.
The corresponding discrete source and eigenvalue problems 
of \eqref{Hodge_mixed} and \eqref{Hodge_eig} are the 
following, respectively:

Given $ \vf \in W^k $, 
find  $ (\vsigma_h, \vu_h)\in V^{k-1}_h\times V^{k}_h$
such that
\begin{equation}\label{dsd_auxiliary_discrete}
	\begin{split}
		\left\langle \vsigma_h \, ,\,  \vtau_h \right\rangle 
		- \left\langle \vu_h \, ,\,  d^{k-1}\vtau_h \right\rangle 
		&=0  \qqq \, \qqq \vtau_h \in V^{k-1}_h,\\
		\left\langle d^{k-1}\vsigma_h \, ,\,  \vv_h \right\rangle 
		+ \left\langle d^k \vu_h \, ,\,  d^k \vv_h \right\rangle 
		+ c \left\langle \vu_h \, ,\,  \vv_h \right\rangle 
		&=  \left\langle \vf \, ,\,  \vv_h \right\rangle \;\,\qq \vv_h \in V^k_h.
	\end{split}
\end{equation}

Find $ (\lambda_h,\vu_h) \in \mbR \times V^k_h $ such that
\begin{equation}\label{Hodge_eig_discrete}
	\begin{split}
		\left\langle \vsigma_h \, ,\,  \vtau_h \right\rangle 
		- \left\langle \vu_h \, ,\,  d^{k-1}\vtau_h \right\rangle 
		&=0 \qqq \qqq \qq \; \vtau_h \in V^{k-1}_h,\\
		\left\langle d^{k-1}\vsigma_h \, ,\,  \vv_h \right\rangle 
		+ \left\langle d^k \vu_h \, ,\,  d^k \vv_h \right\rangle 
		&= \lambda_h \left\langle \vu_h \, \,  \vv_h \right\rangle \qq\, \vv_h \in V^k_h.
	\end{split}
\end{equation}
The finite element spaces are required to have the following properties \cite{arnold2018finite}:
\begin{itemize}
	\item $ 1^\circ $ Approximation property:
	\begin{equation}\nonumber %label{}
		\begin{split}
			\textfor \;\vu \in V^j, \qq
			\lim_{h\to 0}\inf_{\vv_h \in V^j_h}\nm{\vu-\vv_h} = 0,
			\qq  j = k-1 \;\textand\; k.
		\end{split}
	\end{equation}
	\item $ 2^\circ $ Subcomplex property:
	$ d^{k-1} V^{k-1}_h \subset V^k_h $ and $ d^k V^{k}_h \subset V^{k+1}_h $,
	i.e. the three spaces form a complex segment:
	\begin{equation}\label{complex_segment_discrete}
		\begin{split}
			\xymatrix{
				V^{k-1}_h \ar[r]^{d^{k-1}} 
				& V^{k}_h \ar[r]^{d^{k}} 
				& V^{k+1}_h.	
			}
		\end{split}
	\end{equation}
	\item $ 3^\circ $ Bounded cochain projections
	$ \pi^j:V^j \to V^j_h $, $ j = k-1,k,k+1 $:
	the following diagram commutes:
	\begin{equation}\nonumber %label{}
		\begin{split}
			\xymatrix{
				V^{k-1} \ar[r]^{d^{k-1}} \ar[d]^{\pi_h^{k-1}}
				& V^{k} \ar[r]^{d^{k}} \ar[d]^{\pi_h^{k}}
				& V^{k+1} \ar[d]^{\pi_h^{k+1}}\\
				V^{k-1}_h \ar[r]^{d^{k-1}} 
				& V^{k}_h \ar[r]^{d^{k}} 
				& V^{k+1}_h.	
			}
		\end{split}
	\end{equation}
	The projection $ \pi^j_h $ is bounded, 
	i.e. there exists a constant $ c $ such that
	$ \nm{\pi^j_h \vv} \leq c\nm{\vv} $
	for all $ \vv \in V^j $.
	
\end{itemize}
The above three properties can guarantee the convergence of 
the discrete source problem \eqref{dsd_auxiliary_discrete}.
To obtain the correct convergence of 
the discrete eigenvalue problem \eqref{Hodge_eig_discrete},
it needs two additional stronger properties \cite{arnold2010finite}:
\begin{itemize}
	\item $ 4^\circ $ The intersection $ V^k\cap V^*_k $ 
	is a dense subset of $ W^k $ with compact inclusion.
	\item $ 5^\circ $ The cochain projections $ \pi^k_h $ are bound in 
	$ \mcL(W^k,W^k) $
	uniformly with respect to $ h $.
\end{itemize}

\section{The analysis of discrete operators}\label{sec_analysis}

The discrete mixed weak formulations 
\eqref{dsd_auxiliary_discrete} and \eqref{Hodge_eig_discrete}
can be written in the following operator formulations, respectively:

Given $ \vf \in W^k $, 
find  $\vu_h \in  V^{k}_h$ such that
\begin{equation}\label{Hodge_discrete_operator}
	\begin{split}
		(d^k_h)^* d^k_h \vu_h + d^{k-1}_h (d^{k-1}_h)^* \vu_h + c \vu_h =  \vf_h.
	\end{split}
\end{equation}
Here, $ \vf_h $ is the projection of $ \vf $ in $ V^k_h $.

Find $ (\lambda_h,\vu_h) \in \mbR \times V^k_h $ such that
\begin{equation}\label{discrete_operator_eig}
	\begin{split}
		(d^k_h)^* d^k_h \vu_h + d^{k-1}_h (d^{k-1}_h)^* \vu_h = \lambda_h \vu_h.
	\end{split}
\end{equation}

The discrete differential operator $ d^k_h $
is defined as the restriction of $ d^k $ 
on the finite dimensional space $ V^j_h $:
\begin{equation}\label{dh_restriction}
	\begin{split}
		d^k_h = d^k|_{V^k_h}: V^k_h \to V^{k+1}_h.
	\end{split}
\end{equation} 
Then we have 
\begin{equation}\label{ddh0}
	\begin{split}
		d^k_h d^{k-1}_h \equiv d^k d^{k-1} = 0.
	\end{split}
\end{equation}
As the dimension of $ V^j_h $ is finite,
the discrete operator $ d^j_h $ is bounded.
Then its adjoint $ (d^j_h)^* $ is everywhere defined
and the spaces $ V^*_{jh} $ coincide with 
$ W^j_h = V^j_h $.
Then,
for $ \vu_h \in V^{k+1}_h $,
$ (d^k_h)^*\vu_h \in V^k_h $
can be presented as:
\begin{equation}\nonumber %label{}
	\begin{split}
		\left\langle  (d^k_h)^*\vu_h \, ,\, \vv_h \right\rangle 
		=\left\langle \vu_h \, ,\,  d^k_h  \vv_h \right\rangle
		\qq \forall\; \vv \in V^k_h.
	\end{split}
\end{equation}
By the relation \eqref{ddh0}
we have
\begin{equation}\nonumber %label{}
	\begin{split}
		\left\langle (d^{k-1}_h)^*(d^k_h)^*\vu_h \, ,\, \vv_h \right\rangle 
		= \left\langle \vu_h \, ,\, d^k_h d^{k-1}_h\vv_h \right\rangle 
		\equiv \left\langle \vu_h \, ,\, d^k d^{k-1}\vv_h \right\rangle=0
		\qq \forall \; \vu_h, \vv_h \in V^k_h.
	\end{split}
\end{equation}
Then we have 
\begin{equation}\label{dsdsh0}
	\begin{split}
		(d^{k-1}_h)^*(d^k_h)^* = 0.
	\end{split}
\end{equation}
The range and the null spaces of the discrete differential operators
are denotes by
\begin{equation}\label{def_subspace}
	\begin{split}
		\mfB^k_h = d^{k-1}_hV^{k-1}_h,
		\qq \mfZ^k_h = \mcN(d^k_h),
		\qq \mfB^{*}_{kh} = (d^{k}_h)^*V^{k+1}_h
		\qq \text{and}
		\qq \mfZ^{*}_{kh} = \mcN((d^{k-1}_h)^*).
	\end{split}
\end{equation}
The space of discrete harmonic $k$-forms is denoted by
$ \mfH^k_h = \mfZ^k_h \cap \mfZ^{*}_{kh} $.
Then there is the discrete Hodge decomposition \cite[(5.6)]{arnold2018finite}:
\begin{equation}\label{discrete_Hodge_decomposition}
	\begin{split}
		V^k_h = \mfB^k_h \oplus \mfH^k_h \oplus \mfB^{*}_{kh}.
	\end{split}
\end{equation}
By the definitions \eqref{def_subspace},
we have the following results.

\begin{lemma}\label{BN_decomp}
	We have the orthogonal decompositions for $ V^k_h $:
	\begin{equation}\nonumber %label{}
		\begin{split}
			V^k_h = \mfB^{*}_{kh} \oplus \mfZ^k_h
			\qq \text{and}\qq
			V^k_h = \mfB^k_h \oplus \mfZ^{*}_{kh}.
		\end{split}
	\end{equation}
	\end{lemma}
\begin{proof}

	Let $ \left( \mfB^{*}_{kh} \right) ^{\perp} $ 
	be the orthogonal complement of $ \mfB^{*}_{kh} $ in $ V^k_h $.
	We have the orthogonal decomposition
	\begin{equation}\nonumber %label{}
		\begin{split}
			V^k_h = \mfB^{*}_{kh} \oplus \left( \mfB^{*}_{kh} \right) ^{\perp}.
		\end{split}
	\end{equation}
    For $ \vu_h \in \left( \mfB^{*}_{kh} \right) ^{\perp} $,
    we have 
    \begin{equation}\nonumber %label{}
    	\begin{split}
    		\left\langle d^k_h \vu_h \, ,\,   \vv_h \right\rangle
    		= \left\langle \vu_h \, ,\,   (d^k_h)^*\vv_h \right\rangle
    		= 0
    		\qq \text{for any  }  \vv_h \in V^{k+1}_h.
    	\end{split}
    \end{equation}
    Then we have $ d^k_h \vu_h = 0 $, $ \vu_h \in \mfZ^k_h $ 
    and $ \left( \mfB^{*}_{kh} \right) ^{\perp} \subset \mfZ^k_h $.
	
	For $ \vu_h \in \mfZ^k_h $,
	we have $ d^k_h \vu_h = 0 $ and
	\begin{equation}\nonumber %label{}
		\begin{split}
			\left\langle  \vu_h \, ,\,   (d^k_h)^*\vv_h \right\rangle 
			= \left\langle  d^k_h\vu_h \, ,\,  \vv_h \right\rangle 
			= 0
			\qq \text{for any  }  \vv_h \in V^{k+1}_h.
		\end{split}
	\end{equation}
    Then we have $ \vu_h \perp \mfB^{*}_{kh} $, 
    $ \vu_h \in \left( \mfB^{*}_{kh} \right) ^{\perp} $,
    $  \mfZ^k_h \subset \left( \mfB^{*}_{kh} \right) ^{\perp} $
    and $ \left( \mfB^{*}_{kh} \right) ^{\perp} = \mfZ^k_h $.
    
    We finish the proof for the first decomposition.
	The proof for the second one is similar.
	\end{proof}

\begin{theorem}\label{invariant}
	The three subspaces $ \mfB^k_h $, $ \mfH^k_h $ and $ \mfB^{*}_{kh} $
	are invariant subspaces of the operators 
	$ (d^k_h)^* d^k_h $ and $ d^{k-1}_h (d^{k-1}_h)^*  $,
	and there are
	\begin{equation}\nonumber %label{}
		\begin{split}
			(d^k_h)^* d^k_h \mfH^k_h &= \{ 0\}, \qqq \qq
			(d^k_h)^* d^k_h \mfB^{*}_{kh} = \mfB^{*}_{kh},\qqq \;\,\,
			(d^k_h)^* d^k_h \mfB^k_h = \{ 0\}, \\
			d^{k-1}_h (d^{k-1}_h)^* \mfH^k_h &= \{ 0\}, \qq
			d^{k-1}_h (d^{k-1}_h)^* \mfB^{*}_{kh} = \{ 0\}, \qq 
			d^{k-1}_h (d^{k-1}_h)^* \mfB^k_h = \mfB^k_h.
		\end{split}
	\end{equation}
	\end{theorem}

\begin{proof}
	By the definition $ \mfH^k_h = \mfZ^k_h \cap \mfZ^{*}_{kh} $,
	for $ \vu_h \in \mfH^k_h $,
	there is $ d^k_h \vu_h = (d^{k-1}_h)^* \vu_h = 0 $.
	Then we have $ (d^k_h)^* d^k_h \mfH^k_h = \{ 0\} $
	and $ d^{k-1}_h (d^{k-1}_h)^* \mfH^k_h = \{ 0\} $.
	
	By \eqref{ddh0}, \eqref{dsdsh0}
	and the definitions of $ \mfB^k_h $ and $ \mfB^{*}_{kh} $,
	we have $ (d^k_h)^* d^k_h \mfB^k_h = 
	(d^k_h)^* d^k_h d^{k-1}_h V^{k-1}_h = \{ 0\}$
	and 
	$ d^{k-1}_h (d^{k-1}_h)^* \mfB^{*}_{kh}
	= d^{k-1}_h (d^{k-1}_h)^*(d^{k}_h)^*V^{k+1}_h = \{ 0\} $.
	
	By the definition  $ \mfB^{*}_{kh} = (d^{k}_h)^*V^{k+1}_h \subset V^{k}_h $,
	we have $ d^{k}_h\mfB^{*}_{kh} \subset V^{k+1}_h $,
	and then $ (d^k_h)^* d^k_h \mfB^{*}_{kh} \subset \mfB^{*}_{kh} $.
	If there is a $ \vu_h \in \mfB^{*}_{kh} $ 
	and $ \vu_h \perp (d^k_h)^* d^k_h \mfB^{*}_{kh} $,
	we have 
	\begin{equation}\nonumber %label{}
		\begin{split}
			0 = \left\langle  \vu_h \, ,\,   (d^k_h)^* d^k_h\vu_h \right\rangle
			= \left\langle  d^k_h\vu_h \, ,\,  d^k_h\vu_h \right\rangle.
		\end{split}
	\end{equation}
   Then we have $ d^k_h \vu_h = 0 $ and $ \vu_h \in \mfZ^k_h $.
   As $ \mfB^{*}_{kh} $ is orthogonal to $ \mfZ^k_h $ by Lemma \ref{BN_decomp}, 
   there must be $ \vu_h = 0 $.
   Then we obtain that 
   $ (d^k_h)^* d^k_h \mfB^{*}_{kh} = \mfB^{*}_{kh} $.
   The proof for 
   $ d^{k-1}_h (d^{k-1}_h)^* \mfB^k_h = \mfB^k_h $
   is similar.
	\end{proof}

According to the discrete Hodge decomposition \eqref{discrete_Hodge_decomposition},
the eigenpairs of \eqref{discrete_operator_eig}
can be divided into three orthogonal parts:
\begin{equation}\label{discrete_3_eigenpairs}
	\begin{split}
		\left\lbrace \left( 0\, ,\, \vu^{(0)}_{h,i} \right)  \right\rbrace_{i= 1}^{\dim \mfH^k_h},\qq
		\left\lbrace \left( \lambda^{(1)}_{h,i}\, ,\, \vu^{(1)}_{h,i} \right)  \right\rbrace_{i= 1}^{\dim \mfB^{*}_{kh}}
		\; \textand \;
		\left\lbrace \left(\lambda^{(2)}_{h,i}\, ,\, \vu^{(2)}_{h,i} \right)  \right\rbrace_{i= 1}^{\dim \mfB^k_h}.
	\end{split}
\end{equation}
The three subspaces of the discrete Hodge decomposition 
\eqref{discrete_Hodge_decomposition}
can be spanned by the corresponding eigenfunctions:
\begin{equation}\label{discrete_3_spaces}
	\begin{split}
		\mfH^k_h = \textspan \left\lbrace \vu^{(0)}_{h,i} \right\rbrace_{i= 1}^{\dim \mfH^k_h},
		\;
		\mfB^{k*}_h = \textspan \left\lbrace \vu^{(1)}_{h,i} \right\rbrace_{i= 1}^{\dim \mfB^{*}_{kh}} \; \textand\;
		\mfB^k_h = \textspan \left\lbrace  \vu^{(2)}_{h,i}  \right\rbrace_{i= 1}^{\dim \mfB^k_h}
	\end{split}
\end{equation} 
The discrete Hodge eigenvalue problem 
\eqref{Hodge_eig_discrete} can be divided into two eigenvalue problems:
\begin{equation}\nonumber %label{}
	\begin{split}
		(d^k_h)^* d^k_h \vu_h = \lambda_h \vu_h
		\;\;\textand \;\;
		d^{k-1}_h (d^{k-1}_h)^* \vu_h = \lambda_h \vu_h.
	\end{split}
\end{equation}
Substituting the three sets of eigenpairs 
into the two eigenvalue problems,
we have
\begin{equation}\label{dds_dsd_eig_discrete}
	\begin{split}
		(d^k_h)^* d^k_h \vu^{(0)}_{h,i} &= 0,
		\qqq\qqq d^{k-1}_h (d^{k-1}_h)^* \vu^{(0)}_{h,i} = 0
		\qqq\qqq \text{for} \qq i = 1,\cdots, \dim \mfH^k_h,\\
		(d^k_h)^* d^k_h \vu^{(1)}_{h,i} &= \lambda^{(1)}_{h,i} \vu^{(1)}_{h,i}, 
		\qq \, d^{k-1}_h (d^{k-1}_h)^* \vu^{(1)}_{h,i} = 0
		\qqq\qqq \text{for} \qq i = 1,\cdots, \dim \mfB^{*}_{kh},\\
		(d^k_h)^* d^k_h \vu^{(2)}_{h,i} &= 0, 
		\qqq\qqq d^{k-1}_h (d^{k-1}_h)^* \vu^{(2)}_{h,i} = \lambda^{(2)}_{h,i} \vu^{(2)}_{h,i}
		\qq\, \text{for} \qq i = 1,\cdots, \dim \mfB^k_h.\\
	\end{split}
\end{equation}

\section{Auxiliary schemes for the discrete $ d^*d $ problems}\label{sec_auxiliarye}

Our purpose is to solve the corresponding discrete problems
of the following source and eigenvalue problems:
%\begin{align}
%	%
%	(d^k)^* d^k \vu  + c\vu &= \vf \qq\;\;\, \textin \; \Omega, 
%	\qqq \qq\;
%	\tr \star d^k\vu = 0 \qq\;\texton \; \partial \Omega.\label{dd_source}\\  
%	(d^k)^* d^k \vu  &= \lambda \vu \qq \textin \; \Omega,
%	\qqq  \qq \;
%	\tr \star d^k\vu = 0 \qq\;\texton \; \partial \Omega.  \label{dd_eig}
%	%
%\end{align}
\begin{align}
	(d^k)^* d^k \vu  + c\vu &= \vf,  \label{dd_source}\\  
	(d^k)^* d^k \vu  &= \lambda \vu.  \label{dd_eig}
\end{align}
In this section,
we use the  ‘full’ Hodge Laplacian to 
construct auxiliary schemes for 
the problems of
the ‘half’ Laplacian operator $ (d^k)^* d^k $.

\subsection {The auxiliary discrete source problem}

We write the source problem \eqref{dd_source} in weak formulation:
given $ \vf \in W^k $, 
find  $ \vu_h \in V^{k}_h$
such that
\begin{equation}\label{original_weak_discrete}
	\begin{split}
		\left\langle d^k \vu \, ,\,  d^k \vv \right\rangle 
		+ c\left\langle \vu \, ,\,  \vv \right\rangle 
		= \left\langle \vf \, ,\,  \vv \right\rangle \qq \vv \in V^k.		%
	\end{split}
\end{equation}
The discrete weak formulation is:
given $ \vf\in W^k $, find $ \vu_h \in V^k_h $
such that
\begin{equation}\label{dsd_weak_discrete}
	\begin{split}
		\left\langle d^k \vu_h \, ,\,  d^k \vv_h \right\rangle 
		+ c \left\langle \vu_h \, ,\,  \vv_h \right\rangle 
		= \left\langle \vf \, ,\,  \vv_h \right\rangle \qq \vv_h \in V^k_h.
	\end{split}
\end{equation}
%Here, the finite element space $ V^k_h $ on complex \eqref{complex_segment_discrete}.
The operator form is:
given $ \vf\in W^k $, find $ \vu_h \in V^k_h $
such that
\begin{equation}\label{dsd_operator_discrete}
	\begin{split}
		(d^k_h)^* d^k_h \vu_h + c \vu_h = \vf_h,
	\end{split}
\end{equation}
where $ \vf_h $ is the projection of $ \vf $ in $ V^k_h $.

We insert a term $ d^{k-1}_h (d^{k-1}_h)^* $ into the equation \eqref{dsd_operator_discrete}
and obtain an \textbf{auxiliary discrete source problem}: 
given $ \vf\in W^k $, find $ \tilde \vu_h \in V^k_h $ such that
\begin{equation}\label{dsd_auxiliary_opertaor_discrete}
	\begin{split}
		(d^k_h)^* d^k_h \tilde \vu_h + d^{k-1}_h (d^{k-1}_h)^* \tilde \vu_h + c \tilde \vu_h = \vf_h.
	\end{split}
\end{equation}
We have the following theorem.
\begin{theorem}
	If $ \tilde \vu_h $ is the solution of 
	the auxiliary discrete  problem \eqref{dsd_auxiliary_opertaor_discrete}, 
	then
	\begin{equation}\nonumber %label{}
		\begin{split}
			\vu_h = \tilde \vu_h + \frac{1}{c}  d^{k-1}_h (d^{k-1}_h)^* \tilde \vu_h
		\end{split}
	\end{equation}
	is the solution of the original discrete problem \eqref{dsd_operator_discrete}.
\end{theorem}

\begin{proof}
	By the relation $ d^k_h d^{k-1}_h = 0 $ in \eqref{ddh0},
	we inset a zero term 
	\begin{equation}\nonumber %label{}
		\begin{split}
			\frac{1}{c}(d^k_h)^* d^k_h d^{k-1}_h (d^{k-1}_h)^* \tilde \vu_h \equiv 0
		\end{split}
	\end{equation}
	into \eqref{dsd_auxiliary_opertaor_discrete},
	and we have
	\begin{equation}\label{recover_proof}
		\begin{split}
			\vf_h 
			&= (d^k_h)^* d^k_h \tilde \vu_h + d^{k-1}_h (d^{k-1}_h)^* \tilde \vu_h + c \tilde \vu_h\\
			&= (d^k_h)^* d^k_h \tilde \vu_h + \frac{1}{c}(d^k_h)^* d^k_h d^{k-1}_h (d^{k-1}_h)^* \tilde \vu_h 
			+ d^{k-1}(d^{k-1})^* \tilde \vu_h + c \tilde \vu_h\\
			&= (d^k_h)^* d^k_h \left( \tilde \vu_h + \frac{1}{c}d^{k-1}_h (d^{k-1}_h)^* \tilde \vu_h\right) 
			+ c\left( \tilde \vu_h + \frac{1}{c}d^{k-1}_h (d^{k-1}_h)^* \tilde \vu_h \right).
		\end{split}
	\end{equation}
	Comparing \eqref{dsd_operator_discrete} and \eqref{recover_proof},
	we finish the proof. 
	\end{proof}

	We summarize this result in \eqref{auxiliary_source_operator}.
\begin{equation}\label{auxiliary_source_operator}
	\boxed{
		\begin{split}
			& \text{ Solve the equation} \\ 
			& \qq (d^k_h)^* d^k_h \tilde \vu_h + d^{k-1}_h (d^{k-1}_h)^* \tilde \vu_h + c \tilde \vu_h = \vf_h.\\
			& \text{ Then, } \vu_h = \tilde \vu_h + \frac{1}{c}  d^{k-1}_h (d^{k-1}_h)^* \tilde \vu_h\\
			& \text{ is the solution of}\\
			& \qq (d^k_h)^* d^k_h \vu_h + c \vu_h = \vf_h.
		\end{split}
	}
\end{equation}
The auxiliary discrete problem \eqref{dsd_auxiliary_opertaor_discrete} 
is nothing but the Hodge Laplacian problem \eqref{Hodge_discrete_operator}
with the variable $ \vu_h $ being replaced by $ \tilde\vu_h $.
Its weak formulation is:
given $ \vf \in W^k $, 
find  $ (\vsigma_h, \tilde\vu_h)\in V^{k-1}_h\times V^{k}_h$
such that
\begin{equation}\label{dsd_auxiliary_discrete_weak}
	\begin{split}
		\left\langle \vsigma_h \, ,\,  \vtau_h \right\rangle 
		- \left\langle \tilde\vu_h \, ,\,  d^{k-1}\vtau_h \right\rangle 
		&=0  \qqq  \qqq \, \vtau_h \in V^{k-1}_h,\\
		\left\langle d^{k-1}\vsigma_h \, ,\,  \vv_h \right\rangle 
		+ \left\langle d^k \tilde \vu_h \, ,\,  d^k \vv_h \right\rangle 
		+ c \left\langle \tilde\vu_h \, ,\,  \vv_h \right\rangle 
		&=  \left\langle \vf \, ,\,  \vv_h \right\rangle \;\,\qq \vv_h \in V^k_h.
	\end{split}
\end{equation}

\subsection{The auxiliary discrete eigenvalue problem}\label{eig_cases}

We write the $ (d^k)^* d^k $ eigenvalue problem \eqref{dd_eig} in weak formulation:
find $ (\lambda \, ,\, \vu) \in \mbR \times V^k$
such that
\begin{equation}\label{original_eig_weak}
	\begin{split}
		\left\langle d^k \vu \, ,\,  d^k \vv \right\rangle 
		= \lambda \left\langle \vu \, ,\,  \vv \right\rangle \qq \vv \in V^k.
	\end{split}
\end{equation}
The discrete eigenvalue problem is:
find $ (\lambda_h \, ,\, \vu_h) \in \mbR \times V^k_h$
such that
\begin{equation}\label{dsd_eig_weak_discrete}
	\begin{split}
		\left\langle d^k \vu_h \, ,\,  d^k \vv_h \right\rangle 
		= \lambda_h \left\langle \vu_h \, ,\,  \vv_h \right\rangle \qq \vv_h \in V^k_h.
	\end{split}
\end{equation}
This problem can be also written in discrete operator form:
find $ (\lambda_h \, ,\, \vu_h) \in \mbR \times V^k_h$ such that
\begin{equation}\label{dsd_eig_operator_discrete}
	\begin{split}
		(d^k_h)^* d^k_h \vu_h = \lambda_h \vu_h.
	\end{split}
\end{equation}
Similar to the auxiliary discrete source problem 
\eqref{dsd_auxiliary_opertaor_discrete},
we insert the term $ d^{k-1}_h (d^{k-1}_h)^* $ 
into the eigenvalue problem \eqref{dsd_eig_operator_discrete}
and obtain an \textbf{auxiliary discrete eigenvalue problem}: 
find $ (\lambda_h \, ,\, \vu_h) \in \mbR \times V^k_h$
such that
\begin{equation}\label{dsd_eig_auxiliary_operator_discrete}
	\begin{split}
		(d^k_h)^* d^k_h \vu_h + d^{k-1}_h (d^{k-1}_h)^* \vu_h = \lambda_h \vu_h.
	\end{split}
\end{equation}
This auxiliary problem is 
just the weak mixed eigenvalue problem \eqref{Hodge_eig_discrete}.

Let $ (\lambda_h \, ,\, \vu_h) $ be an eigenpair  of 
the auxiliary discrete eigenvalue problem \eqref{dsd_eig_auxiliary_operator_discrete}.
We decompose the eigenvector $ \vu_h $ into three orthogonal parts according to 
the discrete Hodge decomposition \eqref{discrete_Hodge_decomposition}:
\begin{equation}\label{decomp_uh}
	\begin{split}
		\vu_h = \vu_h^{(0)} + \vu_h^{(1)} + \vu_h^{(2)},
	\end{split}
\end{equation}
where $ \vu_h^{(0)} \in \mfH^k_h  $,
$ \vu_h^{(1)} \in \mfB^{*}_{kh} $,
and $ \vu_h^{(2)} \in \mfB^k_h $.
Putting the decomposition into \eqref{dsd_eig_auxiliary_operator_discrete},
we have
\begin{equation}\nonumber %\label{}
	\begin{split}
		(d^k_h)^* d^k_h \vu_h + d^{k-1}_h (d^{k-1}_h)^* \vu_h
		\equiv (d^k_h)^* d^k_h \vu_h^{(1)} + d^{k-1}_h (d^{k-1}_h)^* \vu_h^{(2)}
		&=\lambda_h \left( \vu_h^{(0)} + \vu_h^{(1)} + \vu_h^{(2)} \right),\\
		\underbrace{(d^k_h)^* d^k_h \vu_h^{(1)} - \lambda_h \vu_h^{(1)} }_{\mfB^{*}_{kh}}
		+ \underbrace{d^{k-1}_h (d^{k-1}_h)^* \vu_h^{(2)}- \lambda_h \vu_h^{(2)} }_{\mfB^k_h}
		&=\underbrace{\lambda_h \vu_h^{(0)}}_{\mfH^k_h}.
	\end{split}
\end{equation}
By Theorem \ref{invariant},
the three parts in the above equality belong to three orthogonal subspaces, respectively.
Then, the three parts are linear independent.
To make the equality hold,
the three parts must be all zero.
Then we have the following two conclusions:
\begin{itemize}
	
	\item If $ \vu_h^{(0)}  \not = 0 $, 
	there is that $\lambda_h = 0 $, $ \vu_h^{(1)} = 0 $ and $ \vu_h^{(2)} = 0 $ .
	
	\item If $ \vu_h^{(1)} \not = 0 $ and $ \vu_h^{(2)} \not = 0 $,
    $ \vu_h^{(1)} $ and $ \vu_h^{(2)} $ both are the eigenvectors 
	of the auxiliary discrete eigenvalue problem \eqref{dsd_eig_auxiliary_operator_discrete},
	and they share the same eigenvalue $ \lambda_h \not = 0 $.
	In this case, $ \vu_h^{(0)}  = 0 $

	\end{itemize}

As $ (\lambda_h \, ,\, \vu_h) $ is an eigenpair 
of the auxiliary problem \eqref{dsd_eig_auxiliary_operator_discrete},
not the eigenpair of the original problem 
\eqref{dsd_eig_weak_discrete} or
\eqref{dsd_eig_operator_discrete},
we should judge 
whether this pair belongs to the desired pairs.
By \eqref{dsd_eig_auxiliary_operator_discrete} 
and the decomposition \eqref{decomp_uh},
we have
\begin{equation}\label{auxiliary_recompute_eig}
	\begin{split}
		 \lambda_h 
		&= \frac{\left\langle d^k \vu_h \, ,\,  d^k \vv_h \right\rangle
		+ \left\langle (d^{k-1}_h)^* \vu_h \, ,\,  (d^{k-1}_h)^* \vu_h \right\rangle}
		{\left\langle \vu_h \, ,\,  \vu_h \right\rangle}\\
		&\equiv \frac{\left\langle d^k \vu_h^{(1)} \, ,\,  d^k \vu_h^{(1)} \right\rangle
			+ \left\langle (d^{k-1}_h)^* \vu_h^{(2)} \, ,\,  (d^{k-1}_h)^* \vu_h^{(2)} \right\rangle}
		{	\left\langle \vu_h^{(0)} \, ,\,  \vu_h^{(0)} \right\rangle 
			+ \left\langle \vu_h^{(1)} \, ,\,  \vu_h^{(1)} \right\rangle
			+ \left\langle \vu_h^{(2)} \, ,\,  \vu_h^{(2)} \right\rangle 
		}.
	\end{split}
\end{equation}
We recompute the eigenvalue in the eigenpair 
by the original discrete eigenvalue problem 
\eqref{dsd_eig_weak_discrete}:
\begin{equation}\label{recompute_eig}
	\begin{split}
		\tilde \lambda_h = \frac{\left\langle d^k \vu_h \, ,\,  d^k \vu_h \right\rangle}
		{\left\langle \vu_h \, ,\,  \vu_h \right\rangle} 
		\equiv
		\frac{\left\langle d^k \vu_h^{(1)} \, ,\,  d^k \vu_h^{(1)}\right\rangle}
		{\left\langle \vu_h^{(0)} \, ,\,  \vu_h^{(0)} \right\rangle 
			+ \left\langle \vu_h^{(1)} \, ,\,  \vu_h^{(1)} \right\rangle
			+ \left\langle \vu_h^{(2)} \, ,\,  \vu_h^{(2)} \right\rangle }.
	\end{split}
\end{equation}
By comparing the values $ \lambda_h $ and $ \tilde\lambda_h $,
there are the following cases:
\begin{itemize}
	
	\item If $ \lambda_h  = 0 $, we know 
	that the components $ \vu_h^{(1)} = \vu_h^{(2)}=0 $ in \eqref{decomp_uh}.
	Then we have $ \vu_h \equiv \vu_h^{(0)} \in \mfH^k_h $.
	By \eqref{recompute_eig},
	it is obvious that $ \tilde \lambda_h  = 0 $ in this case.
	
	\item If $ \lambda_h  \not = 0 $ and $ \tilde \lambda_h = \lambda_h $,
	we have $ \vu_h^{(0)} = \vu_h^{(2)}=0 $ 
	and $ \vu_h \equiv \vu_h^{(1)} \in \mfB^{*}_{kh}  $.
	
	\item If $ \lambda_h  \not = 0 $ and $ \tilde \lambda_h = 0 $,
	we have $ \vu_h^{(0)} = \vu_h^{(1)}=0 $ 
	and $ \vu_h \equiv \vu_h^{(2)} \in \mfB^k_h  $.
	
	\item If $ \lambda_h  \not = 0 $
	and $ 0 <\tilde \lambda_h < \lambda_h $,
	we have $ \vu_h^{(0)} = 0 $,
	$ \vu_h^{(1)} \not = 0 $ and $ \vu_h^{(2)} \not = 0 $.
	In this case, $ \vu_h^{(1)} $ and $ \vu_h^{(2)} $ 
	share the same eigenvalue $ \lambda_h $ in 
	the auxiliary discrete eigenvalue problem 
	\eqref{dsd_eig_auxiliary_operator_discrete} 
	or \eqref{auxiliary_recompute_eig}.
	The two components in $ \vu_h = \vu_h^{(1)} + \vu_h^{(2)} $
	can be separated by:
	\begin{equation}\nonumber
		\begin{split}
			(d^k_h)^* d^k_h \vu_h 
			& \equiv (d^k_h)^* d^k_h (\vu_h^{(1)} + \vu_h^{(2)})	
			= (d^k_h)^* d^k_h \vu_h^{(1)} = \lambda_h\vu_h^{(1)} \in \mfB^{*}_{kh},\\
			d^{k-1}_h (d^{k-1}_h)^* \vu_h
			& \equiv d^{k-1}_h (d^{k-1}_h)^* (\vu_h^{(1)} + \vu_h^{(2)})	
			= d^{k-1}_h (d^{k-1}_h)^* \vu_h^{(2)} = \lambda_h\vu_h^{(2)} \in \mfB^k_h .
		\end{split}
	\end{equation}
	
	\end{itemize}
We summarize these cases in Table \ref{recognize_eig_continuous}.
The functions in  $ \mfH^k_h $ and $ \mfB^k_h $
correspond to
the eigenfunctions with $ 0 $ eigenvalues in 
the original discrete eigenvalue problem
\eqref{dsd_eig_weak_discrete} or
\eqref{dsd_eig_operator_discrete},
while the functions in $ \mfB^{*}_{kh} $ 
correspond to the eigenfunction with nonzero eigenvalues.
In actual applications,
the desired eigenpairs are mostly 
the parts in $ \mfH^k_h $ and $ \mfB^{*}_{kh} $.
These eigenpairs can be identified 
by recomputing the eigenvalues and this table.
\begin{table}[ht] 
	\caption{Recognize the type of the eigenpair $ (\lambda_h \, ,\, \vu_h) $
	of the auxiliary eigenvalue problem
	\eqref{dsd_eig_auxiliary_operator_discrete}. }
\label{recognize_eig_continuous}
	\centering  
	\begin{tabular}{|c|c|c| c| }	
		\hline
		$ \lambda_h = 0  $ & \multicolumn{2}{|c|}{$ \vu_h \in \mfH^k_h $} & Type 0\\
		\hline
		\multirow{3}*{$ \lambda_h \not = 0  $}
		& $ \tilde \lambda_h = \lambda_h  $
		&  $ \vu_h \in \mfB^{*}_{kh} $  & Type 1\\
		\cline{2-4}
		~ & $ \tilde \lambda_h = 0 $
		& $ \vu_h\in \mfB^k_h $ & Type 2\\
		\cline{2-4}
		~ & $ 0 <\tilde \lambda_h < \lambda_h $ &
		\tabincell{c}{$ (d^k_h)^* d^k_h  \vu_h \in \mfB^{*}_{kh}  $  \\
			$ d^{k-1}_h (d^{k-1}_h)^*\vu_h \in \mfB^k_h $} &
		Type 3 \\
		\hline
	\end{tabular}
\end{table}
We summarize this auxiliary discrete eigenvalue problem as
\begin{equation}\label{auxiliary_eig_operator}
	\boxed{
		\begin{split}
			& \text{ Find eigenpairs 
				$ \left\lbrace  (\lambda_{h,i} \, ,\, \vu_{h,i})  \right\rbrace_{i}^k $ of}  \\ 
			& \qq (d^k_h)^* d^k_h \vu_h + d^{k-1}_h (d^{k-1}_h)^* \vu_h = \lambda_h \vu_h.\\
			& \text{ Recompute the eigenvalue by }\\
			& \qq \tilde \lambda_{h,i} 
			= \frac{\left\langle d^k \vu_{h,i} \, ,\,  d^k \vu_{h,i} \right\rangle}
			{\left\langle \vu_{h,i} \, ,\, \vu_{h,i} \right\rangle}.\\
			& \text{ Recognize the types of the eigenpairs by 
				Table \ref{recognize_eig_continuous}. }
		\end{split}
	}
\end{equation}

\subsection{Why do we use the auxiliary formulations?}

In the next section, 
we will consider how to deal with
the auxiliary scheme in actual computations.
Before that, we talk about 
the reason why we construct these auxiliary formulations.

If we use direct methods to solve the
corresponding linear system of 
the discrete problem $ (d^k_h)^* d^k_h \vu_h + c \vu_h = \vf_h $,
the auxiliary problem is definitely superfluous.
However, for large-scale problems,
direct methods become inefficient and even impossible.
When using iterative methods, the huge kernel of $ (d^k_h)^* d^k_h $
causes difficulties in using many iterative methods and preconditioning techniques 
that are efficient for Laplace-like problems.
If the complex has \emph{compactness property},
the dimension of the harmonic forms $ \mfH^k $ is limited.
If the discrete complex satisfies the approximation properties in Section \ref{approx_Hodge},
the discrete harmonic from $ \mfH^k_h $ and $ \mfH^k $ are isomorphic 
\cite[Theorem 5.1]{arnold2018finite}.
This means that the dimension of the kernel 
of the discrete Hodge Laplacian 
$ (d^k_h)^* d^k_h  + d^{k-1}_h (d^{k-1}_h)^* $ are at most finite.
Also from the \emph{compactness property},
the spectral distribution of the discrete auxiliary problem 
\eqref{dsd_auxiliary_opertaor_discrete} becomes Laplace-like.

When solving the discrete eigenvalue problem 
$ (d^k_h)^* d^k_h \vu_h = \lambda_h \vu_h $,
the kernel of $ (d^k_h)^* d^k_h $ like a deep hole.
If the eigensolver is chosen thoughtlessly,
it is easy to fall into this hole.
If we use the auxiliary form
$ ((d^k_h)^* d^k_h + d^{k-1}_h (d^{k-1}_h)^*)\vu_h = \lambda_h \vu_h $,
we can compute its eigenvalues 
in the same way as solving Laplace eigenvalue problems.
After computing some eigenpairs of it,
we can identify the desired pairs 
by Table \ref{recognize_eig_continuous}.
Furthermore, 
the preconditioners designed for the auxiliary source problems
is likely valid for the auxiliary eigenvalue problems,
which can reduce the difficulty significantly 
in solving large-scale eigenvalue problems.

%For the discrete eigenvalue problem
%$ (d^k_h)^* d^k_h \vu_h = \lambda \vu_h $, 
%what we are most interested in are the smallest nonzero eigenpairs 
%and the harmonic forms.
%We can compute directly smallest eigenvalue of the auxiliary discrete eigenvalue problem 
%\eqref{dsd_eig_auxiliary_operator_discrete},
%and then recognize the types of the eigenpairs by Table \ref{recognize_eig_continuous}.
%The preconditioning techniques for the auxiliary source problem can be used in
%iterative eigensolvers.
%
%For a system that is not very large,
%the runtime to solve it once is tolerable.
%If solving system need to be repeated many times,
%it is not easy to improve it performance 
%through parallelization
%as the characteristics of direct methods.
%If we use iterative methods to solve it,
%the huge kernel of the operator $ (d^k_h)^* d^k_h $
%become an obstacle as we discussed in the Introduction.
%
%In view of the problems above, 
%we propose the auxiliary forms
%If the complex segment is Fredholm and compact,
%the multiplicities of low-frequency modes 
%of $ (d^k)^* d^k + d^{k-1} (d^{k-1})^*  $
%are finite.
%Then the auxiliary problem
%$   (d^k)^*d^k \tilde \vu + d^{k-1}(d^{k-1})^* \tilde \vu 
%+ c \tilde \vu = \vf_V $
%become a Laplace-like problem.
%Many iterative methods are matching to solve it
%and it is easy to design preconditioners for this type problem.

\section{The matrix forms of the auxiliary schemes }\label{matrix_aux}

In the previous sections,
we construct the auxiliary schemes
for the discrete problems
in operator forms.
In actual computations,
we deal with these discrete problems 
in matrix form.
In this section,
we construct the auxiliary matrix forms
according to the mixed finite element discretizations.

The matrices is generated by the finite element spaces
on the complex segment \eqref{complex_segment_discrete}.
Let $ M = \dim V^{k-1}_h $, $ N = \dim V^{k}_h $,
$ \left\lbrace \vtau_{h,i}\right\rbrace_{i=1}^M  $
and
$ \left\lbrace \vv_{h,i}\right\rbrace_{i=1}^N  $
be the base of the space $ V^{k-1}_h $ and $ V^k_h $,
respectively.
We use light letters to denote
the coefficient vectors of the 
functions in $ V^{k-1} $ and $ V^{k} $.
According to the mixed formulations,
we construct the coefficient matrices
$ \mcA \in \mbC^{N\times N} $ and $ \mcB \in \mbB^{N \times M} $,
the mass matrices
$ \mcM_k \in \mbC^{N\times N} $ and $ \mcM_{k-1} \in \mbB^{M \times M} $
and the right-hand side $ f \in \mbC^N $:
\begin{equation}\nonumber %label{}
	\begin{split}
		\mcA_{ij} &= \left\langle d^k \vv_{h,i} \, ,\, d^k \vv_{h,j} \right\rangle, 
		\qqq \qq\;\;\, \mcB_{ij} = \left\langle d^{k-1}\vtau_{h,j} \, ,\,  \vv_{h,i} \right\rangle,\\
		(\mcM_k)_{ij} &= \left\langle  \vv_{h,i} \, ,\,  \vv_{h,j} \right\rangle,
		\qqq\qq (\mcM_{k-1})_{ij} = \left\langle  \vtau_{h,i} \, ,\,  \vtau_{h,j} \right\rangle, \\
		f_i &= \left\langle  \vf_V \, ,\,  \vv_{h,i} \right\rangle.
	\end{split}
\end{equation}
Then the matrix form  of 
the original discrete source problem \eqref{dsd_weak_discrete} is 
\begin{equation}\label{dsd_matrix}
	\begin{split}
		\mcA u   + c \mcM_k  u &= f.
	\end{split}
\end{equation}
The matrix form of its auxiliary discrete mixed problem 
\eqref{dsd_auxiliary_discrete_weak} is 
\begin{equation}\nonumber %label{}
	\begin{split}
		\mcM_{k-1} \tau - \mcB^T \tilde u &=0, \\
		\mcB \tau + \mcA \tilde u   + c \mcM_k  \tilde u &= f.
	\end{split}
\end{equation}
Here, $ u ,\tilde u \in \mbC^N $ and $ \tau $ 
are the coefficients vectors of $ \vu_h $, $ \tilde \vu_h $ and $ \vtau_h $
in \eqref{dsd_weak_discrete} and \eqref{dsd_auxiliary_discrete_weak}, respectively.
By eliminating the variable $ \tau $,
we obtain the primary form:
\begin{equation}\label{dsd_auxiliary_matrix}
	\begin{split}
		\left( \mcA + \mcB \mcM_{k-1}^{-1} \mcB^T + c \mcM_k\right)  \tilde u = f.
	\end{split}
\end{equation}
The original discrete eigenvalue problem 
\eqref{dsd_eig_operator_discrete} in matrix form is
\begin{equation}\label{orig_eig_matrix}
	\begin{split}
		\mcA u = \lambda_h \mcM_k u.
	\end{split}
\end{equation}
The matrix form of the auxiliary eigenvalue problems 
\eqref{dsd_eig_auxiliary_operator_discrete} is
\begin{equation}\label{orig_eig_matrix_aux}
	\begin{split}
		\left( \mcA + \mcB \mcM_{k-1}^{-1} \mcB^T \right)  u.
		= \lambda_h \mcM_k u.
	\end{split}
\end{equation}

Let $ u^{(0)}_{i} $, $ u^{(1)}_{i} $ and $ u^{(2)}_{i} $
be the coefficients vector of 
$ \vu^{(0)}_{h,i} $, $ \vu^{(1)}_{h,i} $ and $ \vu^{(2)}_{h,i} $ 
in \eqref{discrete_3_spaces}, 
respectively.
Then, consistent with \eqref{discrete_3_eigenpairs},
the eigenpairs can be divided into three parts:
\begin{equation}\label{discrete_3_eigenpairs_vector}
	\begin{split}
		\left\lbrace \left( 0\, ,\, u^{(0)}_{i} \right)  \right\rbrace_{i= 1}^{\dim \mfH^k_h},
		\qq
		\left\lbrace \left( \lambda^{(1)}_{h,i}\, ,\, u^{(1)}_{i} \right)  \right\rbrace_{i= 1}^{\dim \mfB^{*k}_{h}}
		\; \textand \;
		\left\lbrace \left(\lambda^{(2)}_{h,i}\, ,\, u^{(2)}_{i} \right)  \right\rbrace_{i= 1}^{\dim \mfB^k_h}.
	\end{split}
\end{equation}
Also, consistent with \eqref{dds_dsd_eig_discrete},
we have
\begin{equation}\label{eig_relation_matrix}
	\begin{split}
		\mcA u^{(0)}_{i} &= 0, 
		\qqq \qqq \qq\;\; \mcB \mcM_{k-1}^{-1} \mcB^T u^{(0)}_{i} = 0
		\qqq\qqq \qq \text{for} \qq i = 1,\cdots, \dim \mfH^k_h,\\
		\mcA u^{(1)}_{i} &= \lambda^{(1)}_{h,i}\mcM_k u^{(1)}_{i},
		\qq \mcB \mcM_{k-1}^{-1} \mcB^T u^{(2)}_{i} = 0
		\qqq\qqq\qq \text{for} \qq i = 1,\cdots, \dim \mfB^{*}_{kh},\\
		\mcA u^{(2)}_{i} &= 0,
		\qqq\qqq \qq\;\; \mcB \mcM_{k-1}^{-1} \mcB^T u^{(2)}_{i} = \lambda^{(2)}_{h,i}\mcM_k u^2_{i}
		\qq\, \text{for} \qq i = 1,\cdots, \dim \mfB^k_h.\\
	\end{split}
\end{equation}
The eigenvectors in \eqref{eig_relation_matrix} 
can span the three subspace of $ \mbC^N $:
\begin{equation}\label{CN_decomp}
	\begin{split}
		\mbC_0 =\textspan \left\lbrace  u^{(0)}_{i}   \right\rbrace_{i= 1}^{\dim \mfH^k_h},
		\;\;
		\mbC_1 =\textspan \left\lbrace  u^{(1)}_{i}   \right\rbrace_{i= 1}^{\dim \mfB^{k*}_{h}}
		\;\; \textand \;\;
		\mbC_2 &=\textspan \left\lbrace  u^{(2)}_{i}   \right\rbrace_{i= 1}^{\dim\mfB^k_h}.\\
	\end{split}
\end{equation}
According to the discrete Hodge decomposition \eqref{discrete_Hodge_decomposition},
we have the following $ \mcM_k $-orthogonal decomposition
for $ \mbC^N $:
\begin{equation}\label{decomposition_vector}
	\begin{split}
		\mbC^N = \mbC_0 \oplus_{\mcM_k} \mbC_1 \oplus_{\mcM_k} \mbC_2
	\end{split}
\end{equation}
Summarizing \eqref{eig_relation_matrix} and \eqref{CN_decomp},
we have the following theorem.
\begin{theorem}\label{invariant_matrix}
	\begin{equation}\nonumber %label{}
		\begin{split}
			\mcA \mbC_0 &= \{ 0\}, \qqq\qqq\qq
			\mcA \mbC_1 = \mcM \mbC_1, \qq\qqq\qq
			\mcA \mbC_2 = \{ 0\},\\
			\mcB \mcM_{k-1}^{-1} \mcB^T \mbC_0 &= \{ 0\}, \qq
			\mcB \mcM_{k-1}^{-1} \mcB^T \mbC_1 = \{ 0\}, \qq
			\mcB \mcM_{k-1}^{-1} \mcB^T \mbC_2 = \mcM \mbC_2.\\
		\end{split}
	\end{equation}
	\end{theorem}
The following theorem is the matrix representation for
$ (d^k_h)^* d^k_h d^{k-1}_h (d^{k-1}_h)^* $.
\begin{theorem}\label{dd_is_0_matrix_theorem}
	$ \mcA \mcM_{k}^{-1}\mcB \mcM_{k-1}^{-1} \mcB^T = 0. $
\end{theorem}
\begin{proof}
	By Theorem \ref{invariant_matrix}, we have	
	\begin{equation}\nonumber
		\begin{split}
			\mcA  \mcM_{k}^{-1}\mcB \mcM_{k-1}^{-1} \mcB^T \mbC^N 
			&\equiv \mcA\mcM_{k}^{-1}\mcB \mcM_{k-1}^{-1} \mcB^T 
			\left( \mbC_0 \oplus_{\mcM_k} \mbC_1 \oplus_{\mcM_k} \mbC_2\right) \\
			&= \mcA \mcM_{k}^{-1}\mcB \mcM_{k-1}^{-1} \mcB^T \mbC_2\\ 
			&=  \mcA \mbC_2 \\
			&= \{ 0 \}.
		\end{split}
	\end{equation}	
\end{proof}

\subsection{ The modification of $ \mcB^T \mcM_{k-1}^{-1}\mcB $}

As the inverse of the mass matrix $ \mcM_{k-1}^{-1} $
is involved,
the systems of \eqref{dsd_auxiliary_matrix}
and \eqref{orig_eig_matrix_aux}
lose the sparsity.
It is intolerable to
compute the explicit inverse
of a large-sale matrix.
An alternative method is to solve 
a mass equation $ \mcM_{k-1} \tau = \mcB u $
when computing the product  $ \mcB \mcM_{k-1}^{-1} \mcB^T u $
in iterative methods.
Mostly, the conditions of the mass matrices are quite good.
It converges very fast using PCG method.
However,
it is still time-consuming 
if it repeats too many times,
which can happen in iterative solvers.
In addition,
because of this term,
the auxiliary matrix forms
\eqref{dsd_auxiliary_matrix} and \eqref{orig_eig_matrix_aux}
no longer fit into
the solvers that are designed for sparse and explicit matrices.
To deal with this problem,
we give the following theorem.
\begin{theorem}\label{replace_U}
	For a symmetric positive definite (SPD) matrix $ \mcU \in \mbC^{M\times M}$,
	$ \mcB \mcU \mcB^T u = 0 $ if and only if
	$ \mcB \mcM_{k-1}^{-1} \mcB^T u = 0 $
	for  $ u \in \mbC^N $.
\end{theorem}
\begin{proof}
	For  $ u \in \mbC^N $,	
	because $ \mcU $ and $ \mcM_{k-1}^{-1} $ are both SPD matrices,
	we have
	\begin{equation}\nonumber %label{}
		\begin{split}
			\mcB \mcU \mcB^T u = 0 &\Longrightarrow u^T \mcB \mcU \mcB^T u = 0 \\
			& \Longrightarrow (\mcB u)^T \mcU (\mcB u) = 0 \\
			& \Longrightarrow \mcB u = 0\\
			& \Longrightarrow \mcB \mcM_{k-1}^{-1} \mcB^T u = 0.
		\end{split}
	\end{equation}
	The inverse is similar.
\end{proof}
This theorem shows that
the image and the kernel of the matrix $ \mcB^T\mcU \mcB $
stay the same with $ \mcB \mcM_{k-1}^{-1} \mcB^T $.
Then the results in Theorem \ref{invariant_matrix}
still hold if we replace $ \mcM_{k-1}^{-1} $
by a SPD matrix $ \mcU $.
The following is a direct corollary of 
Theorem \ref{dd_is_0_matrix_theorem}
with this modification.
\begin{corollary}\label{modified_U}
	$ \mcA  \mcM_{k}^{-1}\mcB \mcU \mcB^T = 0. $
\end{corollary}
\begin{proof}
	
	As the matrix $ \mcB \mcU \mcB^T $ is Hermitian,
	the space $ \mbC^N $ can be decomposed
	as the orthogonal sum of the  image and kernel of $ \mcB \mcU \mcB^T $,
	i.e.
	$$ \mbC^N = \IM \mcB \mcU \mcB^T \oplus \Ker \mcB \mcU \mcB^T. $$
	Similarly, there is
	$$ \mbC^N = \IM \mcB \mcM_{k-1}^{-1} \mcB^T \oplus \Ker \mcB \mcM_{k-1}^{-1} \mcB^T. $$
	By Theorem \ref{replace_U},
	we have 
	$$ \Ker \mcB \mcU \mcB^T = \Ker \mcB \mcM_{k-1}^{-1} \mcB^T.$$
	Then we obtain
	$$ \IM \mcB \mcU \mcB^T = \IM \mcB \mcM_{k-1}^{-1} \mcB^T,$$
	i.e.
	\begin{equation}\nonumber %label{}
		\begin{split}
			\mcB \mcU \mcB^T \mbC^N = \mcB \mcM_{k-1}^{-1} \mcB^T \mbC^N.
		\end{split}
	\end{equation}
	By Theorem \ref{dd_is_0_matrix_theorem}, we have
	\begin{equation}\nonumber %label{}
		\begin{split}
			\mcA  \mcM_{k}^{-1}\mcB \mcU \mcB^T \mbC^N 
			\equiv\mcA  \mcM_{k}^{-1}\left( \mcB \mcU \mcB^T \mbC^N\right) 
			= \mcA  \mcM_{k}^{-1} \left( \mcB \mcM_{k-1}^{-1} \mcB^T \mbC^N\right)
			\equiv \mcA  \mcM_{k}^{-1} \mcB \mcM_{k-1}^{-1} \mcB^T \mbC^N = \{ 0 \} .
		\end{split}
	\end{equation}
\end{proof}
In \cite{mixed_vector}, 
the $ 1 $-form and $ 2 $-form Hodge Laplacian problems on $ \mbR^3 $ complex
are solved using a multigrid method.
In the auxiliary problems
 \eqref{dsd_auxiliary_matrix}
and \eqref{orig_eig_matrix_aux}, 
if we replace $ \mcM_{k-1}^{-1} $ by a SPD matrix $ \mcU $,
the auxiliary parts still satisfy the property in Corollary \ref{modified_U}.
If the matrix $ \mcU $ is chosen properly and has enough sparsity,
the systems of the auxiliary problems can be sparse enough
and be easier to solve using iterative methods.

\subsection{Auxiliary scheme for the source problem}

We modify the auxiliary equation \eqref{dsd_auxiliary_matrix} as
\begin{equation}\label{dsd_auxiliary_matrix_modified}
	\begin{split}
		\left( \mcA + \mcB \mcU \mcB^T + c \mcM_k\right)  \tilde u = f.
	\end{split}
\end{equation}
We have the following result.
\begin{theorem}
	If $ \tilde u $ is the solution of the auxiliary equation
	\eqref{dsd_auxiliary_matrix_modified},
	then 
	\begin{equation}\label{solution_origi}
		\begin{split}
			u = \tilde u + \frac{1}{c}\mcM_{k}^{-1}\mcB \mcU \mcB^T \tilde u
		\end{split}
	\end{equation}
	is the solution of the original problem $ \mcA u   + c \mcM_k  u = f $.
\end{theorem}
\begin{proof}
	By Corollary  \ref{modified_U}, 
	we insert a zero term 
	\begin{equation}\nonumber %label{}
		\begin{split}
			\frac{1}{c}\mcA\mcM_{k}^{-1}\mcB\mcU\mcB^T \tilde u \equiv 0
		\end{split}
	\end{equation}
	into \eqref{dsd_auxiliary_matrix_modified}
	and we have
	\begin{equation}\nonumber %label{}
		\begin{split}
			f &=
			\left( \mcA + \mcB \mcU \mcB^T + c \mcM_k\right)  \tilde u \\
			&= \mcA \tilde u
			+ \frac{1}{c}\mcA\mcM_{k}^{-1}\mcB\mcU\mcB^T \tilde u
			+ \mcB\mcU\mcB^T \tilde u
			+ c\mcM_k\tilde u \\
			&= \mcA\left(\tilde u + \frac{1}{c}\mcM_{k}^{-1}\mcB\mcU\mcB^T \tilde u \right) 
			+ c \mcM_{k}\left(\tilde u + \frac{1}{c}\mcM_{k}^{-1}\mcB\mcU\mcB^T \tilde u \right). 
		\end{split}
	\end{equation}
\end{proof}

If there is a proper $ \mcU $,
the distribution of the spectrum of \eqref{dsd_auxiliary_matrix_modified}
can be similar to \eqref{dsd_auxiliary_matrix}, which is Laplace-like.
As $ \mcM_{k-1} $ is a SPD mass matrix and has quite good condition number,
$ \mcU $ can be choose as a positive number in some cases.
The solution \eqref{solution_origi} involves solving a
equation $ \mcM_k^{-1} \left( \mcB \mcU \mcB^T \tilde u\right) $.
As this is a mass equation,
it is  easy to solve using 
many iterative methods, 
even without any preconditioning.
Based the analyses above, 
we summarize the auxiliary matrix scheme for  
\eqref{dsd_matrix} as
\begin{equation}\label{auxiliary_source}
	\boxed{
		\begin{split}
			& \text{ Set a proper $ \mcU $ and choose an iterative} \\ 
			& \text{ method to solve the equation}\\
			& \qqq \left( \mcA  +  \mcB \mcU \mcB^T + c \mcM_k \right) \tilde u = f.\\
			&\text{ Solve the mass equation }\\
			&\qqq \mcM_k v =  \mcB \mcU \mcB^T \tilde u.\\
			& \text{ Then, } u = \tilde u + \frac{1}{c} v
			 \text{  is the solution of} \\
			& \qqq\left( \mcA + c \mcM_k \right) u = f.
		\end{split}
	}
\end{equation}

As we solve the auxiliary problem \eqref{auxiliary_source}
using iterative methods,
the solutions are not mathematically exact,
and they contain numerical errors.
Letting $ \tilde u_a $ and the $ v_a $ be the 
approximate solution in the auxiliary problem
\eqref{auxiliary_source},
we denote their residuals by 
\begin{equation}\nonumber %label{}
	\begin{split}
		r_{aux} &= f - \left( \mcA  +  \mcB \mcU \mcB^T + c \mcM_k \right) \tilde u_a,\\
		r_{mass} &= \mcB \mcU \mcB^T \tilde u_a - \mcM_k v_a.
	\end{split}
\end{equation}
The approximate solution of the original problems is
\begin{equation}\nonumber %label{}
	\begin{split}
		u_a = \tilde u_a + \frac{1}{c} v_a.
	\end{split}
\end{equation}
The residual of $ u_a $ is
	\begin{equation}\label{error_relation}
	\begin{split}
		r_{orig} &= f - \left( \mcA  + c \mcM_k \right) u_a \\
		&\equiv f - \left( \mcA  + c \mcM_k \right) \left( \tilde u_a + \frac{1}{c} v_a \right)  \\
		&= f - \left( \mcA  + c \mcM_k \right)
		\left(  \tilde u_a + \frac{1}{c} \mcM_k^{-1}\mcB \mcU \mcB^T \tilde u_a  - \frac{1}{c}\mcM_k^{-1}r_{mass}\right)\\
		&= f - \left( \mcA  + c \mcM_k \right)
		\left(  \tilde u_a + \frac{1}{c} \mcM_k^{-1}\mcB \mcU \mcB^T \tilde u_a\right)  +\left( \mcA  + c \mcM_k \right)\left(  \frac{1}{c}\mcM_k^{-1}r_{mass}\right)\\
		&= f - \left( \mcA  +  \mcB \mcU \mcB^T + c \mcM_k \right) \tilde u_a
		+\left( \mcA  + c \mcM_k \right)\left(  \frac{1}{c}\mcM_k^{-1}r_{mass}\right)\\
		&= r_{aux} + \left(\frac{1}{c}\mcA \mcM_k^{-1} + 1 \right) r_{mass}.
	\end{split}
\end{equation}
Here, we use the result in Corollary \ref{modified_U}.
From this relation,
we find that if the error of the mass equation is small enough,
the error of the original equation can be close to 
the error of the auxiliary equation \eqref{dsd_auxiliary_matrix_modified}.
Fortunately,
the error of the mass equation can be easily reduced 
using iterative methods
because of its good condition.

\subsection{Auxiliary scheme for the eigenvalue problem}

We modify the auxiliary eigenvalue problem \eqref{orig_eig_matrix_aux} as
\begin{equation}\label{dsd_eig_auxiliary_matrix}
	\begin{split}
		\left( \mcA + \mcB\mcU\mcB^T \right)  u = \lambda_h \mcM_k  u.
	\end{split}
\end{equation}
After we obtain an eigenpair $ \left(\lambda_{h}\, ,\, u \right) $
of the auxiliary scheme \eqref{dsd_eig_auxiliary_matrix},
we recompute the eigenvalue through the original scheme \eqref{orig_eig_matrix}:
\begin{equation}\label{recompute_matrix}
	\begin{split}
		\tilde \lambda_h = \frac{u^T \mcA u}{u^T \mcM_k u}.
	\end{split}
\end{equation}
We decompose the eigenvector $ u $ into three $ \mcM $-orthogonal parts
according to \eqref{decomposition_vector}:
\begin{equation}\nonumber %label{}
	\begin{split}
		u = u^{(0)} + u^{(1)} + u^{(2)},
	\end{split}
\end{equation}
where $ u^{(0)} \in \mbC_0 $,
$ u^{(1)} \in \mbC_1 $
and $ u^{(2)} \in \mbC_2 $.
Similar to the analyses in Section \ref{eig_cases},
the eigenpairs can be divided into the following cases.
\begin{itemize}
	
	\item If $ \lambda_h = 0 $, we know that $ u^{(1)} = u^{(2)} = 0 $
	and $ u = u^{(0)} \in \mbC_0  $.
	
	\item If $ \lambda_h \not = 0 $ and $ \tilde \lambda_h = \lambda_h $,
	we have $ u^{(0)} = u^{(2)} = 0 $
	and $ u = u^{(1)} \in \mbC_1 $.
	
	\item If $ \lambda_h \not = 0 $ and $ \tilde \lambda_h = 0 $,
	we have $ u^{(0)} = u^{(1)} = 0 $
	and $ u = u^{(2)} \in \mbC_2 $. 
	
	\item If $ \lambda_h \not = 0 $ and $ 0< \tilde\lambda_h  < \lambda_h $,
	we have $ u^{(0)} = 0 $, $ u^{(1)} \not= 0 $ and $ u^{(2)} \not= 0 $.
	In this case, $ u = u^{(1)} + u^{(2)} $ 
	and the two components $ u^{(1)} $ and $ u^{(2)} $
	share the same eigenvalue $ \lambda_h $ 
	in the auxiliary problem \eqref{dsd_eig_auxiliary_matrix}.
	The two components can be separated by 
	\begin{equation}\nonumber %label{}
		\begin{split}
			\mcM_k^{-1}\mcA u 
			&\equiv \mcM^{-1}\mcA \left( u^{(1)} + u^{(2)}\right)  
			= \mcM_k^{-1}\mcA u^{(1)} = \lambda_h u^{(1)} \in \mbC_1,\\ 
			\mcM_k^{-1}\mcB\mcU\mcB^T u 
			&\equiv \mcM^{-1}\mcB \mcU \mcB^T \left( u^{(1)} + u^{(2)}\right) 
			= \mcM_k^{-1}\mcB \mcU \mcB^T u^{(2)} = \lambda_h u^{(2)}\in \mbC_2\\
			\textor \qq u^{(2)} 
			&= u - u^{(1)} 
			\equiv u - \frac{1}{\lambda_h} \mcM_k^{-1}\mcA u \in \mbC_2.
		\end{split}
	\end{equation}
	
	\end{itemize}
When involving the signs '$ = $' and '$ < $' for $ \lambda_h $ and $ \tilde \lambda_h $,
numerical errors should be taken into consideration.
These cases are summarized in Table \ref{recognize_eig}.
\begin{table}[ht] % \label{recognize_eig}	
	\caption{Recognize the type the eigenpair $ \left(\lambda_{h}\, ,\, u \right) $
		of the auxiliary matrix form \eqref{dsd_eig_auxiliary_matrix}
		through its recomputed eigenvalue $ \tilde \lambda_h $. }
	\label{recognize_eig}
	\centering  
	\begin{tabular}{|c|c|c| c| }	
		\hline
		$ \lambda_h = 0  $ & \multicolumn{2}{|c|}{$ u \in \mbC_0 $} & Type 0\\
		\hline
		\multirow{3}*{$ \lambda_h \not = 0  $}
		& $ \tilde \lambda_h = \lambda_h $
		& $ u\in\mbC_1 $  & Type 1\\
		\cline{2-4}
		~ & $ \tilde \lambda_h = 0 $
		&  $ u\in\mbC_2 $ & Type 2\\
		\cline{2-4}
		~ & $ 0< \tilde \lambda_h < \lambda_h $ &
		\tabincell{c}{$ \mcM_k^{-1}\mcA u\in\mbC_1 $  \\ $ \mcM_k^{-1}\mcB\mcU\mcB^T u\in\mbC_2 $} &
		Type 3 \\
		\hline
	\end{tabular}
\end{table}
Finally,
we summarize the auxiliary scheme
for eigenvalue problem into the framework:
\begin{equation}\label{auxiliary_eig}
	\boxed{
		\begin{split}
			& \text{Auxiliary scheme for $\mcA  u = \lambda_h \mcM_k u  $:}\\
			& \text{ Set a proper $ \mcU $ and choose an eigensolver} \\ 
			& \text{  to find
				a set of eigenpairs $ \left\lbrace \left(\lambda_{h,i} \, ,\, u_i \right)  \right\rbrace  $}\\
			& \qqq\left( \mcA + \mcB\mcU\mcB^T \right)  u = \lambda_h \mcM_k u.\\
			&\text{ Recompute the eigenvalues with}\\
			&\qqq \tilde\lambda_{h,i} =(u_i^T \mcA u_i)/(u_i^T \mcM_k u_i).\\
			&\text{ Recognize the type of the eigenpairs by Table \ref{recognize_eig}}.\\
		\end{split}
	}
\end{equation}
Because of the modification for $ \mcM_{k-1}^{-1} $,
we know that the eigenpairs in $ \mbC_2 $
are not the exact nonzero eigenpairs of 
\eqref{dsd_auxiliary_matrix}.
However, 
in usual applications,
we are mainly interested in the
eigenvectors in $ \mbC_0 $ and $ \mbC_1 $.

\section{Two preconditioners}\label{sec_precond}

Preconditioning is an important aspect in iterative methods.
A good preconditioner can make the solver stable and faster.
The Laplace problems are often taken as test examples
when a new preconditioner is proposed.
As the  analyses in the previous sections,
the spectral distribution of the auxiliary schemes
\eqref{auxiliary_source} and \eqref{auxiliary_eig}
are Laplace-like.
The preconditioning techniques for algebraic systems 
can be roughly divided into two families:
based the algebraic systems straightforward
and based on the corresponding continuous problems.
In section, 
we introduce one method for each family.
For the sake of simplicity, we denote
$$  \msA = \mcA + \mcB \mcU \mcB^T + c\mcM_k.$$

\subsection{Incomplete LU factorization and sparse approximate inverse preconditioner}

Incomplete LU (ILU) factorization is a type of preconditioners
for general matrices.
The conventional ILU factors are generated by Gaussian elimination method
on a sparse pattern of the origin matrix.
There are many variations of the conventional ILU
to improve its efficiency and parallelism \cite{MR1990645}.
%In \cxx, a type of ILU based on fixed-point iteration is constructed
%and the generation of the ILU factors has fine-grain parallelism.
In our recent paper \cite{IterILU},
we propose an iterative ILU factorization in matrix form,
which provides a parallel way to generate the ILU factors.
In the application of ILU factors in iterative methods,
each iteration involves solving two triangular systems
$ x = U\setminus(L\setminus b) $.
Forward and backward substitution method are effective 
in solving triangular systems,
but they are highly sequential,
which is a restriction in parallelization.
In our another recent paper \cite{SAIT},
we proposed a type of sparse approximate inverse 
for triangular matrices based on Jacobi iteration.
Then the ILU factors $ (L\, ,\,U) $ can be replaced 
by their approximate inverses $ (M_L\, ,\, M_U) $,
and the preconditioning procedure becomes 
two matrix-vector products $ x = M_U(M_L b) $,
which are of fine-grained parallelism.
Furthermore,
if we choose proper parameter,
the number of nonzeros in $ (M_L\, ,\, M_U) $
can be reduced.
We present the generation of 
the sparse aproximate inverse 
preconditioner in Algorithm \ref{SAIT_Thr}.
\begin{algorithm}  
	\caption{Threshold-based Sparse Aproximate Inverse 
		for Triangular matrix $ T $, SAIT\_Thr$ (\tau, m) $}  
	\begin{algorithmic}[1]  
		\STATE	Let $ D $ be the diagonal matrix of the triangular matrix $ T $
		\STATE	Set the initial data $ M_T=I $,
		and let $  T_0 = I-D^{-1}T $.		
		\FOR { $ k = 1, 2, \cdots,m $}
		\STATE  $ M_T =  T_0 M_T + I $
		\STATE  drop the entries in  $ M $ whose magnitudes are small than $ \tau $ $ (\tau<1) $
		\ENDFOR
		\STATE Output $ M_T = M_T D^{-1} $
	\end{algorithmic}  
	\label{SAIT_Thr}
\end{algorithm}

\subsection{Multigrid method}

The multigrid method can be a solver itself
and also can act as preconditioner in other iterative methods.
As it is discussed in the Introduction,
the main difficulty in applying  multigrid method
of the original problem is the kernel of $ \mcA $.
In the auxiliary schemes,
$ \mcB \mcU \mcB^T $ complements this kernel.
For the matrix form 
\begin{equation}\nonumber
	\begin{split}
		\left( \mcA + \mcB \mcM_{k-1}^{-1} \mcB^T + c \mcM_k\right)  \tilde u = f,
	\end{split}
\end{equation}
it is the discrete problem of the equation
\begin{equation}\nonumber
	\begin{split}
		\left( (d^k)^* d^k +  d^{k-1}(d^{k-1})^*\right)  \vu  + c\vu = \vf.
	\end{split}
\end{equation}
%On different levels of geometric multigrid mesh,
%their low frequencies are similar.
%However,
We replace $ \mcM_{k-1}^{-1} $ by 
a symmetric positive definite matrix $ \mcU $ in \eqref{auxiliary_eig}.
To coarse meshes to correct the low frequencies on fine meshes,
the spectral distribution of $ \mcU $ and $ \mcM_{k-1}^{-1} $ 
should be similar.
Algorithm \ref{MGVC} is a V-cycle multigrid method.
\begin{algorithm}  
	\caption{Multigrid V-cycle  $ u_l = MGVC_l(f_l\, , \, u_l^0) $}  
	\begin{algorithmic}[1]  
		\STATE	Set proper $ \mcU $ for each $ \msA_l $ 
		$  $ and give an initial guess $ u_l^0 $ for $ u_l $
		\IF{$ l =1 $}
		\STATE	$ u_1 = \msA_1^{-1} f_1 $
		\ELSE
		\STATE Presmoothing: $ u_l = \mcS_l(u_l\, ,\,f_l \, ,\,\msA_l)$
		\STATE Correction on coarse grid:
		\begin{equation}\nonumber %label{}
			\begin{split}
				e_{k-1} &= \mcR_{l-1}(f_l - \msA_l u_l)\\
				e_{k-1} &= MGVC_{l-1}(e_{k-1} \, , \, 0)\\
				u_l &= u_l + \mcP_l e_{l-1}
			\end{split}
		\end{equation}
		\STATE Postsmoothing: $ u_l = \mcS_l(u_l\, ,\,f_l \, ,\,\msA_l)$
		\ENDIF
	\end{algorithmic}  
	\label{MGVC}
\end{algorithm}

%\clearpage
\section{Numerical experiments}\label{sec_numerical}

In this section, 
we take  Maxwell and grad-div problems as examples to verify 
the auxiliary schemes we proposed in previous sections.
We compute  their source problems
\begin{align}
	\nabla \times \nabla \times \vu + \vu  &= \vf 
	\qq \text{in} \;\; \Omega, 
	\qqq \vn \times \left(\nabla \times \vu\right)  = 0 
	\qq\texton \;\; \partial\Omega, \label{Maxwell_source}\\
	-\nabla(\nabla \cdot \vu) + \vu &= \vf 
	\qq \text{in} \;\; \Omega, 
	\qq\;\;\qqq \vn \left( \nabla \cdot \vu\right)  = 0
	\qq\texton \;\; \partial\Omega,\label{graddiv_source}
\end{align}
and their eigenvalue problems
\begin{align}
	\nabla \times \nabla \times \vu &= \lambda \vu
	\qq \text{in} \;\; \Omega, 
	\qqq \vn \times \left(\nabla \times \vu\right)  = 0 
	\qq\texton \;\; \partial\Omega, \label{Maxwell_eigenvalue}\\
	-\nabla(\nabla \cdot \vu)  &= \lambda \vu
	\qq \text{in} \;\; \Omega, 
	\qq\;\;\qqq \vn \left( \nabla \cdot \vu\right)  = 0
	\qq\texton \;\; \partial\Omega.\label{graddiv_eigenvalue}
\end{align}
The two operators are the  $ k=1 $ and $ k=2 $ forms of the $ d^* d $ operator 
on the $ \mbR^3 $ complex, respectively:
\begin{equation}\nonumber %\label{R3_complex}
	\begin{split}
		\xymatrix{
			0 \ar[r] 
			& H^1(\Omega) \ar[r]^{\nabla}        \ar[d]^{\Pi^Q_h}
			& \vH(\textcurl,\Omega) \ar[r]^{\nabla \times} \ar[d]^{\Pi^\vE_h}
			& \vH(\textdive,\Omega) \ar[r]^{\nabla \cdot}  \ar[d]^{\Pi^\vF_h}
			& L^2(\Omega) \ar[r]                  \ar[d]^{\Pi^S_h}
			& 0\\
			0\ar[r] 
			& Q_h \ar[r]^{\nabla} 
			& \vE_h \ar[r]^{\nabla \times} 
			& \vF_h \ar[r]^{\nabla \cdot}	
			& S_h \ar[r]
			& 0.
		}
	\end{split}
\end{equation}
We use the first order of the first family of \Nedelec elements \cite{nedelec1980mixed}
to discrete these problems
and construct their auxiliary schemes.
The node element space $ Q_h $ 
and the edge element space $ \vE_h $ 
are used to construct the auxiliary schemes for the Maxwell problems,
while $ \vE_h $ and the face element space $ \vF_h $ 
are used for the grad-div problems.

We compute these problems on two three-dimensional domains,
the first one is a cube $ [0,\pi]^3 $
and the second one is also a cube but with a tunnel inside
$ [0,\pi]^3 \setminus [\pi/4, 3\pi/4]^2 \times [0,\pi]$,
shown in Figure \ref{Domain}.
\begin{figure}
	\centering
	\begin{minipage}{6cm}
		\includegraphics[height=4cm,width=5.5cm]{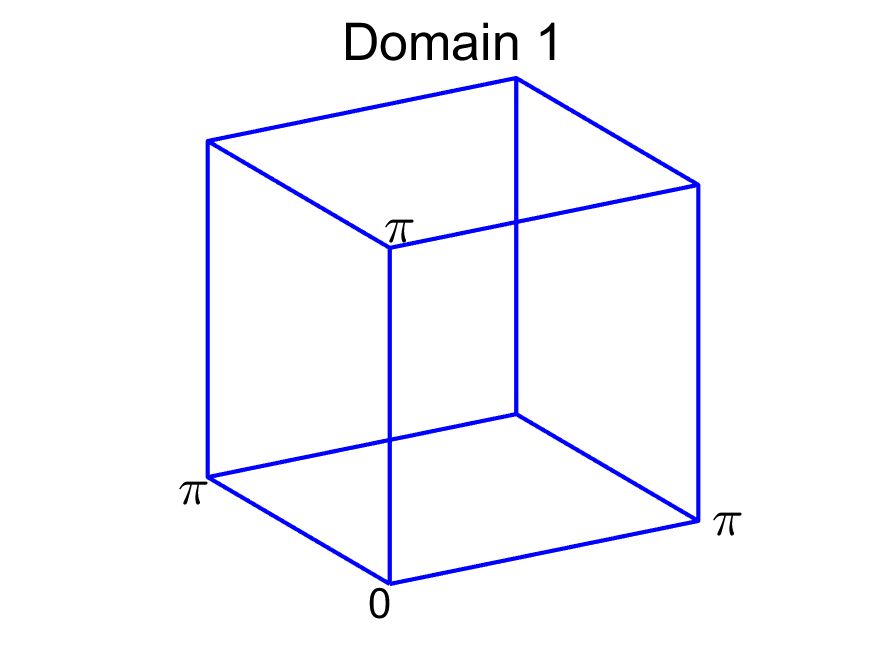}
	\end{minipage}
	\begin{minipage}{6cm}
		\includegraphics[height=4cm,width=5.5cm]{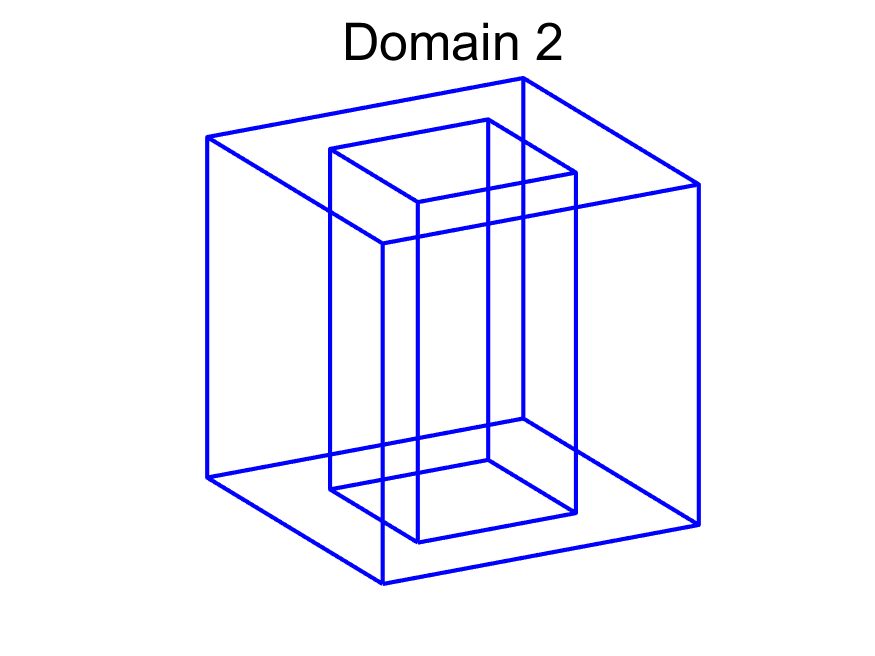}
	\end{minipage}
	\caption{Left: a cubic domain $ [0,\pi]^3 $.
		Right: a cubic domain with a hole $ [0,\pi]^3 \setminus [\pi/4, 3\pi/4]^2 \times [0,\pi]$.}
	\label{Domain}
\end{figure}
We use 5 levels of uniform cubic meshes
and their mesh sizes are
$$ h = \frac{\pi}{4}, \frac{\pi}{8},\frac{\pi}{16},\frac{\pi}{32},\frac{\pi}{64}.$$
For each mesh,
the $ \mcU $ in the auxiliary parts
is set to a diagonal matrix
\begin{equation}\label{U_setting}
	\begin{split}
		\mcU = \frac{5}{h^3} \mcI,
	\end{split}
\end{equation} 
where $ h $ is the mesh size and $ \mcI $ is the identity matrix.

The numerical experiments involve many cases 
including
two different operators,
both source and eigenvalue problems,
two different domains.
Moreover,
the multigrid method acts 
both as an independent solver itself
and a preconditioner in iterative solvers.
To guarantee all these cases to converge,
we choose symmetric Gauss-Seidel method as the smoother 
and set the number of smoothing iteration to $ 5 $ uniformly,
even though the these settings are not optimal for every case.
We generate the SAIT preconditioner based on 
the triangular factors of the standard ILU(0) factorization.
We also take a uniform setting for the parameters in Algorithm \ref{SAIT_Thr},
say SAIT$(0.05, 10)$.

We use the 
Locally Optimal Block Preconditioned Conjugate Gradient method (LOBPCG) \cite{knyazev2001toward}
as the eigensolver to compute 
the first $ 20 $ eigenvalues of the auxiliary eigenvalue problems.
We compute $ 25 $ eigenvalues in actual computations
to make the solver stabler.
When the errors of first $ 20 $ numerical eigenvalue reach the error criterion,
we stop the iteration.

The errors of the linear equations and generalized matrix eigenvalue problems
are set as
\begin{equation}\nonumber %label{}
	\begin{split}
		\frac{\nm{f - \msA u}}{\nm{ f }}
		\qq \text{and} \qq
		\frac{\nm{\msA u - \lambda \mcM_k u}_{\mcM_k}}{\nm{u}_{\mcM_k}},
	\end{split}
\end{equation} 
respectively.
Here,  $ \nm{\cdot}_{\mcM_k} $ denotes the norm with respect to $ \mcM_k $,
that is
\begin{equation}\nonumber %label{}
	\begin{split}
		\nm{v}_{\mcM_k} = \sqrt{v^{T}\mcM_k v}.
	\end{split}
\end{equation} 
The stop criterion is set to $ 10^{-8} $
for all the numerical experiments.

%\clearpage

\subsection{The Maxwell problems}

The discrete weak form of the Maxwell source problem \eqref{Maxwell_source} is:
given $ \vf\in \vL^2(\Omega) $, find $ \vu_h \in \vE_h $ such that
\begin{equation}\label{weak_discrete_Maxwell}
	\begin{split}
		\left\langle \nabla \times \vu_h \, ,\,  \nabla \times \vv_h \right\rangle 
		+ c \left\langle \vu_h \, ,\,  \vv_h \right\rangle 
		= \left\langle \vf \, ,\,  \vv_h \right\rangle \qq \vv_h \in \vE_h.
	\end{split}
\end{equation}
The corresponding auxiliary weak form is:
given $ \vf \in \vL^2(\Omega) $, 
find  $ (\vsigma_h, \tilde\vu_h)\in Q_h \times \vE_h $
such that
\begin{equation}\label{weak_discrete_Maxwell_mixed}
	\begin{split}
		\left\langle \sigma_h \, ,\,  \tau_h \right\rangle 
		- \left\langle \tilde\vu_h \, ,\,  \nabla\tau_h \right\rangle 
		&=0  \qqq  \qqq \;\, \tau_h \in Q_h,\\
		\left\langle \nabla\sigma_h \, ,\,  \vv_h \right\rangle 
		+ \left\langle \nabla \times \tilde \vu_h \, ,\,  \nabla \times \vv_h \right\rangle 
		+ \left\langle \tilde\vu_h \, ,\,  \vv_h \right\rangle 
		&=  \left\langle \vf \, ,\,  \vv_h \right\rangle \;\,\qq \vv_h \in \vE_h.
	\end{split}
\end{equation}
The discrete form of Maxwell eigenvalue problem \eqref{Maxwell_eigenvalue} is:
find $ (\lambda_h,\vu_h) \in \mbR \times \vE_h $ such that
\begin{equation}\label{weak_discrete_Maxwell_eig}
	\begin{split}
		\left\langle \nabla \times \vu_h \, ,\,  \nabla \times \vv_h \right\rangle 
		= \lambda_h \left\langle \vu_h \, ,\,  \vv_h \right\rangle \qq \vv_h \in \vE_h.
	\end{split}
\end{equation}
Its auxiliary problem in weak form is:
find $ (\lambda_h,\vu_h) \in \mbR \times \vE_h $ such that
\begin{equation}\label{weak_discrete_Maxwell_eig_mixed}
	\begin{split}
		\left\langle \sigma_h \, ,\,  \tau_h \right\rangle 
		- \left\langle \vu_h \, ,\,  \nabla\tau_h \right\rangle 
		&=0  \qqq  \qqq \qqq\;\, \tau_h \in Q_h,\\
		\left\langle \nabla\sigma_h \, ,\,  \vv_h \right\rangle 
		+ \left\langle \nabla \times  \vu_h \, ,\,  \nabla \times \vv_h \right\rangle  
		&=  \lambda_h \left\langle \vu_h \, ,\;  \vv_h \right\rangle \;\,\qq \vv_h \in \vE_h.
	\end{split}
\end{equation}
We use the weak forms \eqref{weak_discrete_Maxwell_mixed} and
\eqref{weak_discrete_Maxwell_eig_mixed}
to construct the corresponding auxiliary matrix problems of 
\eqref{auxiliary_source} and \eqref{auxiliary_eig}.
The mass matrix generated by
$ \left\langle \sigma_h \, ,\,  \tau_h \right\rangle $ 
in the above auxiliary forms
is replace by $ \mcU $ as \eqref{U_setting}.

We test the auxiliary schemes in solving the
discrete Maxwell source problem \eqref{weak_discrete_Maxwell} 
on the mesh of level 5 ($ h = \pi/64 $).
The sizes of the discrete systems of 
Domain $ 1 $ and Domain $ 2 $
are $ 811,200 $ and $ 620,736 $, respectively.
We use  Algorithm \ref{SAIT_Thr} to generate 
the SAIT preconditioners based on the standard ILU(0) factorization.
Figure \ref{SAIT_Maxwell} shows the comparisons of 
the nonzeros in the lower triangular factor $ L $
and its corresponding sparse approximate inverse matrix $ M_L $
generated by SAIT$(0.05, m)$.
Here $ m $ is the iteration number.
The results show that the nonzeros on the SAIT$(0.05, 10)$ matrix 
is much less than its original matrix.
\begin{figure}
	\centering
	\begin{minipage}{6cm}
		\includegraphics[height=5cm,width=6.5cm]{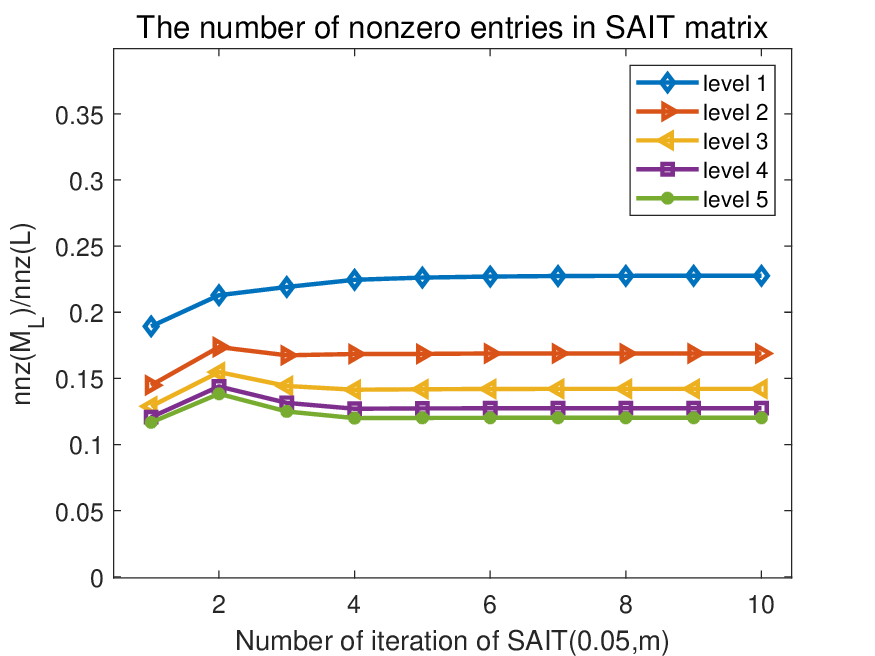}
	\end{minipage}
	\begin{minipage}{6cm}
		\includegraphics[height=5cm,width=6.5cm]{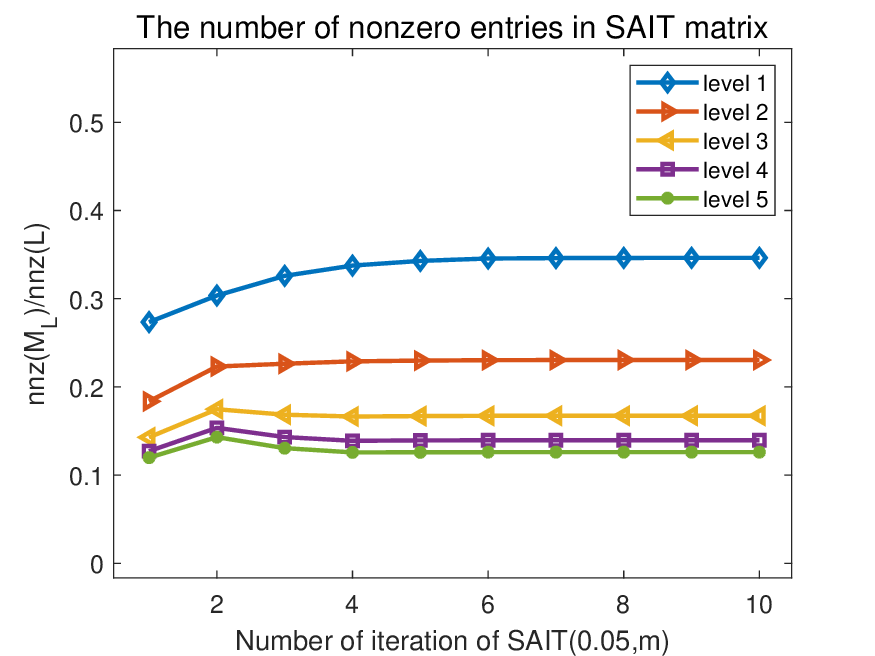}
	\end{minipage}
	\caption{The number of the nonzero entries in the SAIT matrices 
		for the auxiliary discrete Maxwell source problems.
		Left: Domain 1. Right: Domain 2. }
	\label{SAIT_Maxwell}
\end{figure}

Figure \ref{Maxwell_L5} shows the results in solving 
the original discrete Maxwell equation and its auxiliary problem
using the PCG method on the mesh of level 5.
With the auxiliary part, 
the convergence becomes stable and faster,
and can be accelerated by ILU and SAIT preconditioners further.
\begin{figure}
	\centering
	\begin{minipage}{6cm}
		\includegraphics[height=5cm,width=6.5cm]{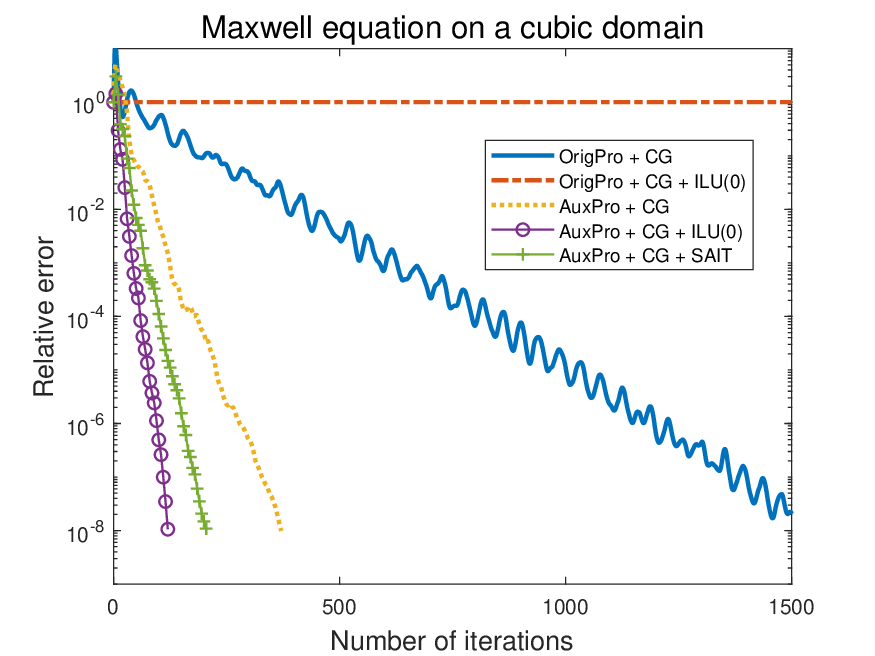}
	\end{minipage}
	\begin{minipage}{6cm}
		\includegraphics[height=5cm,width=6.5cm]{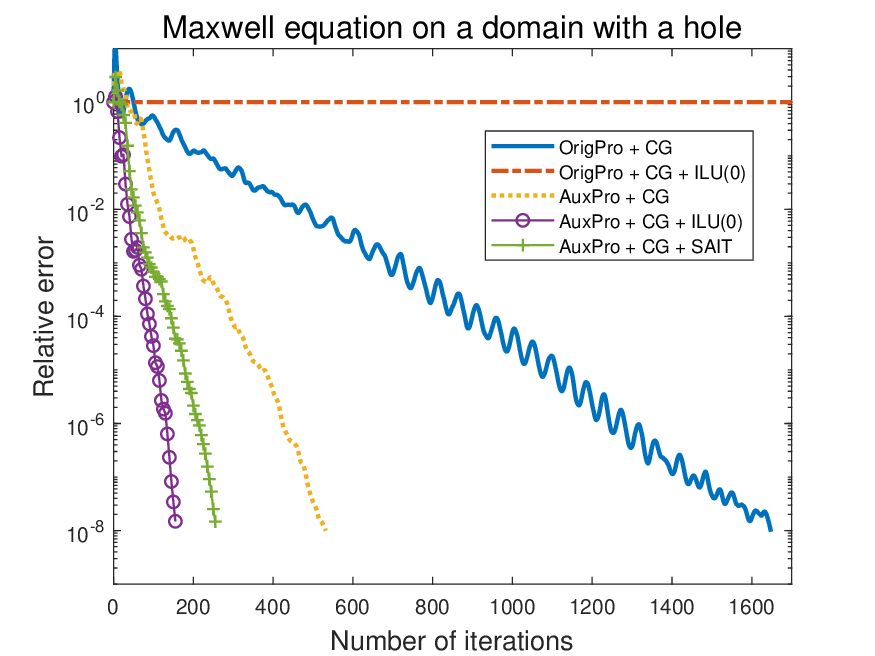}
	\end{minipage}
	\caption{The convergence histories of the original discrete Maxwell equation 
		and its auxiliary scheme with different iterative methods. 
		The mesh is level 5 with the mesh size $ h = \pi/64 $.
		Left: Domain 1. Rigth: Domain 2.}
	\label{Maxwell_L5}
\end{figure}
After getting the solution of the auxiliary problem, 
it needs to solve a mass equation.
Figure \ref{mass_Maxwell} shows the relative errors of the mass equation
and the original discrete Maxwell problem.
The error of the solution of the original problem 
stops going down after some iterations,
which can be explained by 
the relation of the residuals \eqref{error_relation}.
The final solution contains the error in  
the solution of the auxiliary problem.
The figures also show that
the relative error of the mass equation is smaller than 
that of the original equation.
This means that
the error of the mass euqation 
should be set smaller in actual computations.
\begin{figure}
	\centering
	\begin{minipage}{6cm}
		\includegraphics[height=5cm,width=6.5cm]{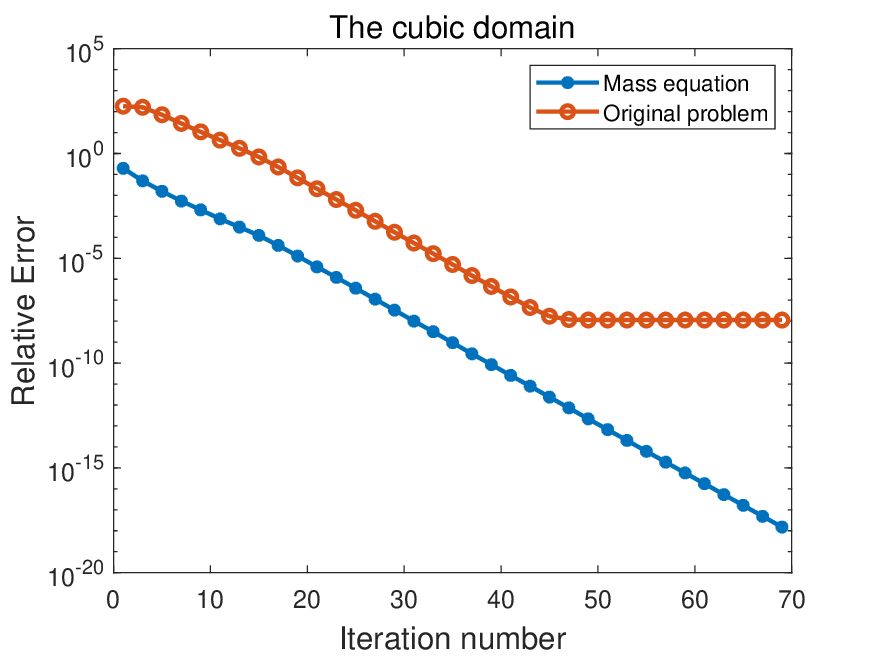}
	\end{minipage}
	\begin{minipage}{6cm}
		\includegraphics[height=5cm,width=6.5cm]{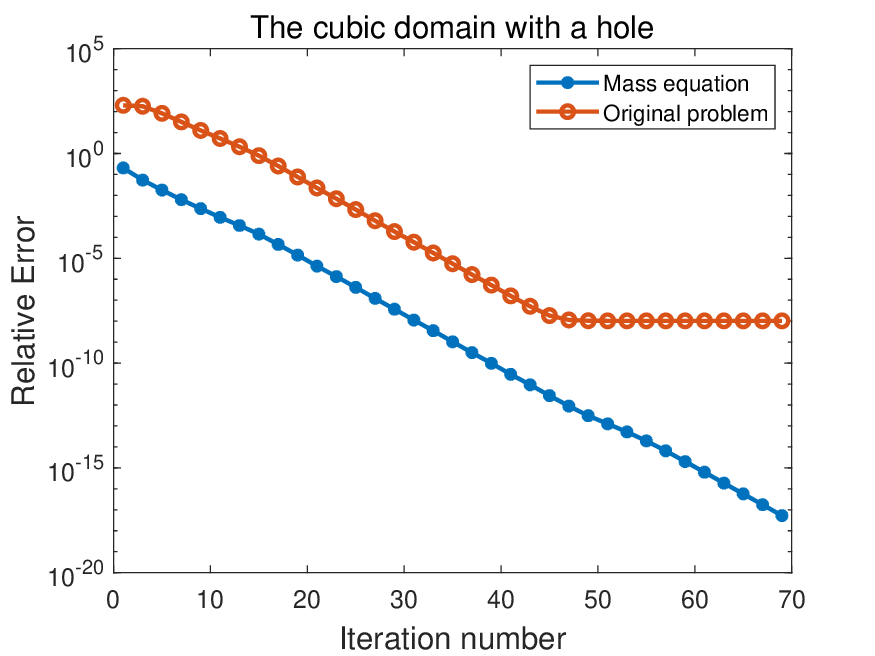}
	\end{minipage}
	\caption{The convergence histories of the mass equation.
		The relative error is computed in two ways: by the mass equation
		and by the original discrete Maxwell equation.
		Left: Domain 1. Rigth: Domain 2.}
	\label{mass_Maxwell}
\end{figure}

Table \ref{Table_Maxwell_Domain_1} and \ref{Table_Maxwell_Domain_2} 
show more convergence results 
and their comparisons
of the discrete Maxwell source problems
on the two domains.
The V-cycle multigrid method works both as an independent solver
or a preconditioner in the PCG method.
It can reduce the iteration counts significantly.
The mass equations on different meshes converge fast even without preconditioning.
The ILU(0) preconditioner can make them converge to the stop criterion within several iterations.
\begin{table}[ht]  %\label{ }
	\footnotesize	
	\caption{ The iteration counts of the original discrete Maxwell equation,
		its auxiliary equation and the mass equation with different iterative methods
		and preconditioners on Domain 1.}
	\label{Table_Maxwell_Domain_1}
	\centering  
	\begin{tabular}{|c|c|c||c|c||c|c|c|c|c||c|c|}	
		\hline
		\multicolumn{3}{|c||}{ } &
		\multicolumn{2}{|c||}{ $ (\mcA + c\mcM_k)u = f $ } &
		\multicolumn{5}{|c||}{ $ (\mcA + \mcB \mcU \mcB^T + c\mcM_k) \tilde u = f $ } &
		\multicolumn{2}{|c|}{ $ \mcM_k v =  \mcB \mcU \mcB^T \tilde u $ } \\
		\hline
		\multirow{2}*{Level} & \multirow{2}*{$ h $} & \multirow{2}*{d.o.f.} &
		\multicolumn{2}{|c||}{\tabincell{c}{Conjugate \\ Gradient}  } &
		\multicolumn{4}{|c|}{ Conjugate  Gradient } & \multirow{2}*{ \tabincell{c}{Multigrid \\ \tiny{MGVC$(l,5)$}} } &
		\multicolumn{2}{|c|}{ \tabincell{c}{Conjugate \\ Gradient}  }\\
		\cline{4-9} \cline{11-12}
		~ & ~ & ~ & $ -- $ & ILU(0) &
		$ -- $ & ILU(0) & SAIT & \tiny{MGVC$(l,5)$} & ~ & $ -- $ & ILU(0) \\
		\hline
		$ l =  1 $ & $ \pi/4 $ & $ 300 $ & $ 96 $  & $ 19 $  &
		$ 64 $ & $ 7 $ & $ 13 $ &   &   &
		$ 15 $ &  $ 3 $ \\
		\hline
		$ l =  2 $ & $ \pi/8 $ & $ 1994 $ & $ 208 $  & $ 2233 $  &
		$ 84 $ & $ 14 $ & $ 23 $ & $ 8 $ & $ 5 $ &
		$ 28 $ &  $ 3 $ \\
		\hline
		$ l =  3 $ & $ \pi/16 $ & $ 13,872 $ & $ 395 $  & $ \infty $  &
		$ 113 $ & $ 29 $ & $ 48 $ & $ 10 $ & $ 7 $ &
		$ 32 $ &  $ 3 $ \\
		\hline
		$ l =  4 $ & $ \pi/32 $ & $ 104,544 $ & $ 769 $  & $ \infty $  &
		$ 198 $ & $ 60 $ & $ 96 $ & $ 12 $ & $ 10 $ &
		$ 33 $ &  $ 3 $ \\
		\hline
		$ l =  5 $ & $ \pi/64 $ & $ 811,200 $ & $ 1544 $  & $ \infty $  &
		$ 372 $ & $ 122 $ & $ 191 $ & $ 15 $ & $ 14 $ &
		$ 33 $ &  $ 3 $ \\
		\hline
	\end{tabular}
	
\end{table}

\begin{table}[ht]  %\label{ }
	\caption{ The iteration counts of the original discrete Maxwell equation,
		its auxiliary equation and the mass equation with different iterative methods
		and preconditioners on Domain 2. }
	\label{Table_Maxwell_Domain_2}
	\footnotesize	
	\centering  
	\begin{tabular}{|c|c|c||c|c||c|c|c|c|c||c|c|}	
		\hline
		\multicolumn{3}{|c||}{ } &
		\multicolumn{2}{|c||}{ $ (\mcA + c\mcM_k) \tilde u = f $ } &
		\multicolumn{5}{|c||}{ $ (\mcA + \mcB \mcU \mcB^T + c\mcM_k) u = f $ } &
		\multicolumn{2}{|c|}{ $ \mcM_k v =  \mcB \mcU \mcB^T \tilde u $ } \\
		\hline
		\multirow{2}*{Level} & \multirow{2}*{$ h $} & \multirow{2}*{d.o.f.} &
		\multicolumn{2}{|c||}{\tabincell{c}{Conjugate \\ Gradient}  } &
		\multicolumn{4}{|c|}{ Conjugate  Gradient } & \multirow{2}*{ \tabincell{c}{Multigrid \\ \tiny{MGVC$(l,5)$}} } &
		\multicolumn{2}{|c|}{ \tabincell{c}{Conjugate \\ Gradient}  }\\
		\cline{4-9} \cline{11-12}
		~ & ~ & ~ & $ -- $ & ILU(0) &
		$ -- $ & ILU(0) & SAIT & \tiny{MGVC$(l,5)$} & ~ & $ -- $ & ILU(0) \\
		\hline
		$ l =  1 $ & $ \pi/4 $ & $ 276 $ & $ 106 $  & $ 20 $  &
		$ 64 $ & $ 9 $ & $ 11 $ &   &   &
		$ 28 $ &  $ 6 $ \\
		\hline
		$ l =  2 $ & $ \pi/8 $ & $ 1,656 $ & $ 219 $  & $ 136 $  &
		$ 93 $ & $ 16 $ & $ 30 $ & $ 5 $ & $ 6 $ &
		$ 35 $ &  $ 7 $ \\
		\hline
		$ l =  3 $ & $ \pi/16 $ & $ 11,184 $ & $ 416 $  & $ \infty $  &
		$ 158 $ & $ 36 $ & $ 57 $ & $ 6 $ & $ 7 $ &
		$ 36 $ &  $ 3 $ \\
		\hline
		$ l =  4 $ & $ \pi/32 $ & $ 81,504 $ & $ 820 $  & $ \infty $  &
		$ 288 $ & $ 78 $ & $ 115 $ & $ 9 $ & $ 11 $ &
		$ 34 $ &  $ 5 $ \\
		\hline
		$ l =  5 $ & $ \pi/64 $ & $ 620,736 $ & $ 1643 $  & $ \infty $  &
		$ 321 $ & $ 157 $ & $ 235 $ & $ 12 $ & $ 16 $ &
		$ 33 $ &  $ 5 $ \\
		\hline
	\end{tabular}
	
\end{table}

We compute the auxiliary Maxwell eigenvalue problem using LOBPCG method 
with different preconditioners.
Figure \ref{LOBPCG_Maxwell} shows 
the convergence histories of the first 20 eigenvalues
on Domain 1 with the mesh of level 5 ($ h = \pi/64 $).
All of the ILU(0), SAIT$(0.05, 10)$ and MGVC$(l,5)$ 
preconditioners can efficiently reduce
the iteration counts.
After getting the eigenpairs of the auxiliary problems,
we recompute the numerical eigenvalues using 
the original eigenvalue problem by \eqref{recompute_matrix}.
Table \ref{recompute_Domain_1} shows that 
the desired smallest nonzero eigenvalues on Domain 1
can be easily identified through recomputing.
Table \ref{recompute_Domain_2} show the eigenvalues on Domain 2.
As there is a tunnel in this domain,
the Betti number is $ 1 $.
The first eigenpair in this auxiliary eigenvalue problem
corresponds to the only one harmonic form in the cohomology space.
Table \ref{Maxwell_eig_domain_1} and \ref{Maxwell_eig_domain_2} 
show the convergences of the auxiliary eigenvalue problems
with different preconditioner on the two domains.
\begin{figure}
	\centering
	\begin{minipage}{6cm}
		\includegraphics[height=5cm,width=6.5cm]{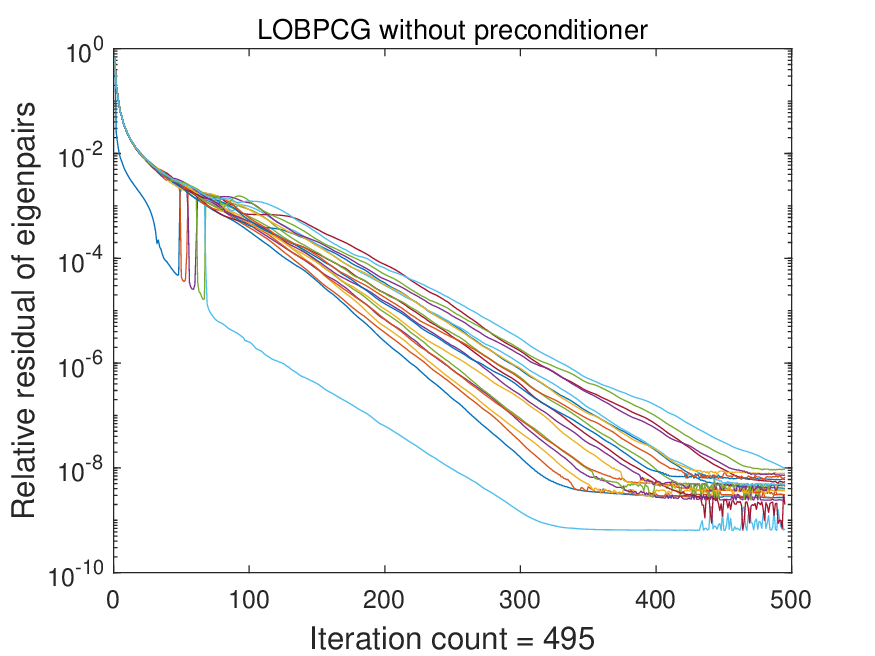}
	\end{minipage}
	\begin{minipage}{6cm}
		\includegraphics[height=5cm,width=6.5cm]{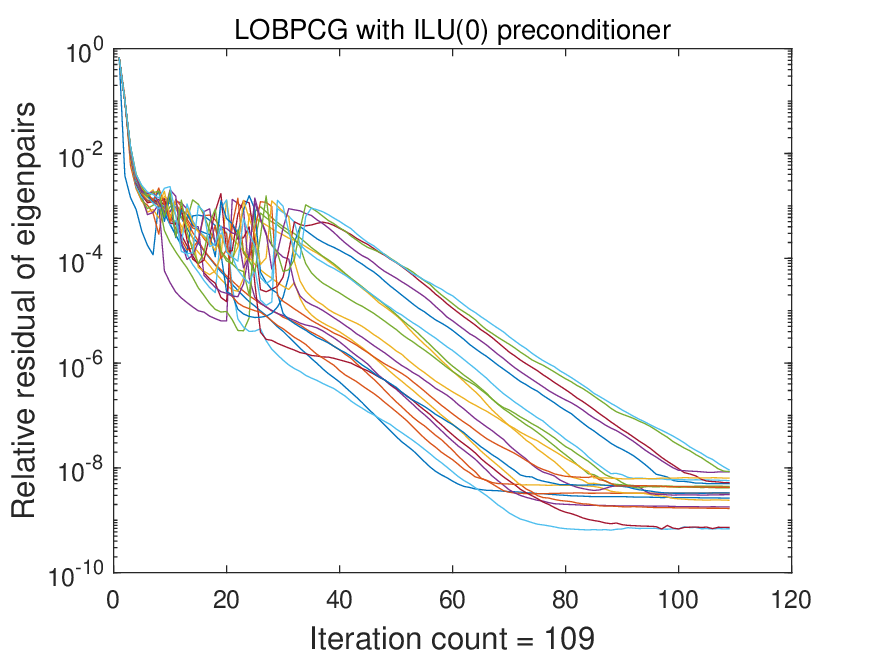}
	\end{minipage}
	
	\begin{minipage}{6cm}
		\includegraphics[height=5cm,width=6.5cm]{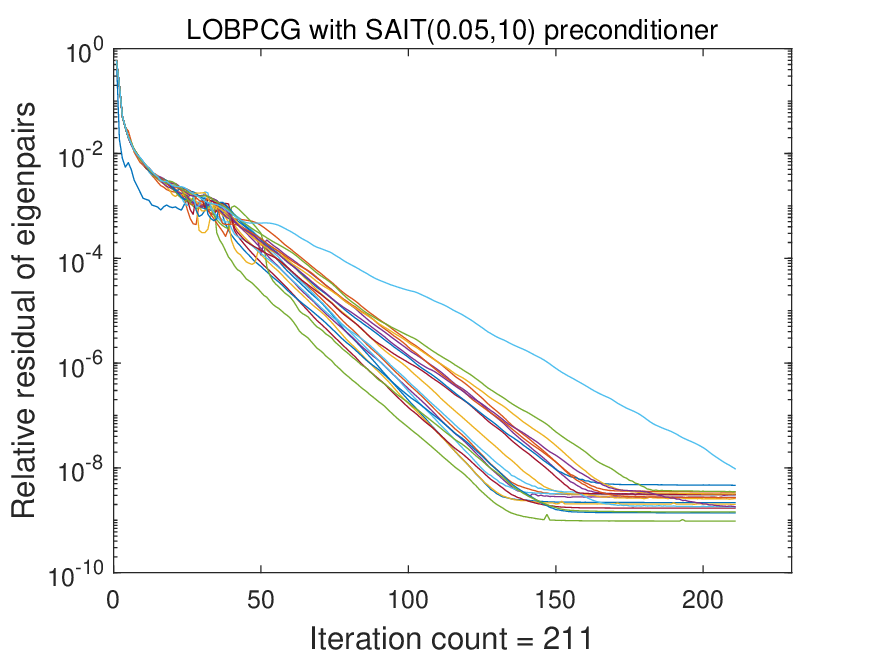}
	\end{minipage}
	\begin{minipage}{6cm}
		\includegraphics[height=5cm,width=6.5cm]{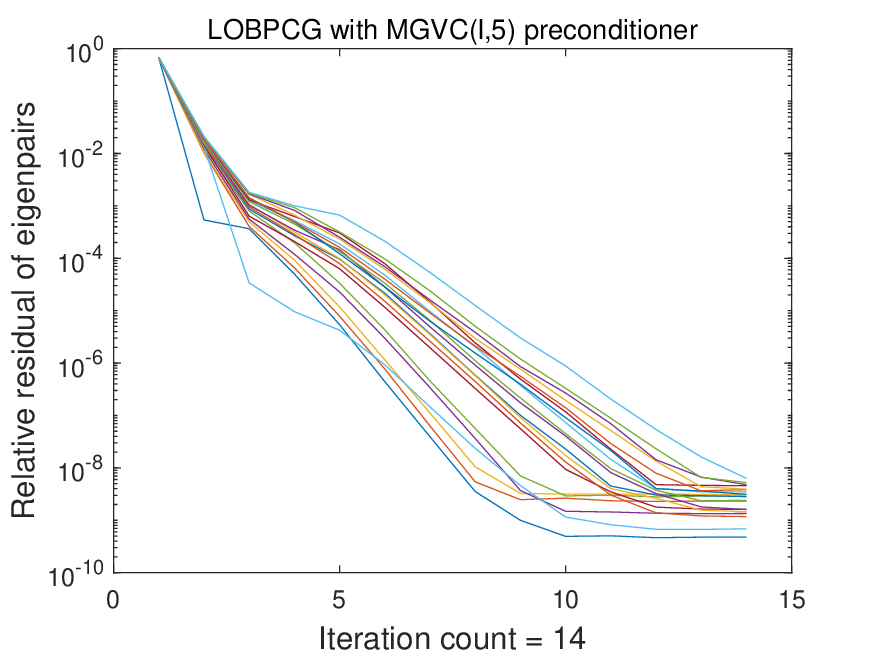}
	\end{minipage}
	\caption{Compute the first 20 eigenvalue of the auxiliary discrete Maxwell operator
		on Domain 1 with the mesh of level 5 ($ h = \frac{\pi}{64} $)
		using LOBPCG method with ILU(0), SAIT$ (0.05,10) $, Multigrid preconditioner, respectively.}
	\label{LOBPCG_Maxwell}
\end{figure}

\begin{table}[ht]  %\label{ }	
	\caption{Recompute the numerical eigenvalues using 
		the original discrete Maxwell eigenproblem
		and its auxiliary eigenproblem on Domain 1, respectively.}
\label{recompute_Domain_1}
	\footnotesize
	\centering  
	\begin{tabular}{|c|c|c|c|c| }	
		\hline
		Exact	&  $\left( \mcA + \mcB \mcU  \mcB^T \right)  u = \lambda_h \mcM_k u$
		& $ \mcA u = \lambda_h \mcM_k u$ 
		&  $ \mcB^T \mcU  \mcB u = \lambda_h \mcM_k u$ & Type\\
		\hline	
		2	&   2.000401   &   2.000401   &   2.781$\times 10^{-14}$ & Type 1 \\
		2	&	2.000401   &   2.000401   &   5.524$\times 10^{-13}$ & Type 1\\
		2	&	2.000401   &   2.000401   &   1.081$\times 10^{-12}$ & Type 1\\
		3	&	3.000602   &   3.000602   &   1.859$\times 10^{-13}$ & Type 1\\
		3	&	3.000602   &   3.000602   &   8.440$\times 10^{-13}$ & Type 1\\
		&	4.696564   &   1.112$\times 10^{-12}$  &   4.696564 & Type 2\\
		&	4.696564   &   2.115$\times 10^{-12}$  &   4.696564 & Type 2\\
		&	4.696564   &   5.233$\times 10^{-12}$  &   4.696564 & Type 2\\
		5	&	5.003414   &   5.003414   &   9.978$\times 10^{-14}$ & Type 1\\
		5	&	5.003414   &   5.003414   &   1.792$\times 10^{-13}$ & Type 1\\
		5	&	5.003414   &   5.003414   &   1.851$\times 10^{-13}$ & Type 1\\
		5	&	5.003414   &   5.003414   &   3.573$\times 10^{-13}$ & Type 1\\
		5	&	5.003414   &   5.003414   &   1.213$\times 10^{-12}$ & Type 1\\
		5	&	5.003414   &   5.003414   &   2.020$\times 10^{-12}$ & Type 1\\
		6	&	6.003615   &   6.003615   &   2.552$\times 10^{-12}$ & Type 1\\
		6	&	6.003615   &   6.003615   &   1.438$\times 10^{-12}$ & Type 1\\
		6	&	6.003615   &   6.003615   &   1.405$\times 10^{-12}$ & Type 1\\
		6	&	6.003615   &   6.003615   &   1.995$\times 10^{-12}$ & Type 1\\
		6	&	6.003615   &   6.003615   &   2.526$\times 10^{-12}$ & Type 1\\
		6	&	6.003615   &   6.003615   &   3.177$\times 10^{-12}$ & Type 1\\
		\hline		
	\end{tabular}
\end{table}
\begin{table}[ht]  %\label{ }
\caption{Recompute the numerical eigenvalues using 
	the original discrete Maxwell eigenproblem
	and its auxiliary eigenproblem on Domain 2, respectively.
    The first eigenfunction is the harmonic form on this domain.}
\label{recompute_Domain_2}
	\footnotesize	
	\centering  
	\begin{tabular}{|c|c|c|c| }	
		\hline
		$\left( \mcA + \mcB \mcU  \mcB^T \right)  u = \lambda_h \mcM_k u$
		& $ \mcA u = \lambda_h \mcM_k u$ 
		&  $ \mcB \mcU \mcB^T u = \lambda_h \mcM_k u$ & Type\\
		\hline	
		5.247$\times 10^{-12}$  &   4.816$\times 10^{-12}$  &     4.970$\times 10^{-13}$ & Type 0\\
		1.000200  &     1.000200   &    9.858$\times 10^{-13}$ & Type 1\\
		1.513696  &     1.513696   &    5.930$\times 10^{-13}$ & Type 1\\
		1.513696  &     1.513696   &    7.567$\times 10^{-13}$ & Type 1\\
		2.375882  &     1.761$\times 10^{-12}$  &     2.375882 & Type 2\\
		2.375882  &     7.735$\times 10^{-12}$  &     2.375882 & Type 2\\
		2.689091  &     2.689091   &   2.327$\times 10^{-12}$  & Type 1\\
		3.400482  &     3.400482   &    4.662$\times 10^{-13}$ & Type 1\\
		4.003213  &     4.003213   &    5.860$\times 10^{-13}$ & Type 1\\
		4.516709  &     4.516709   &    5.356$\times 10^{-13}$ & Type 1\\
		4.516709  &     4.516709   &     2.487$\times 10^{-12}$& Type 1\\
		4.560227  &     3.255$\times 10^{-12}$  &     4.560227 & Type 2\\
		5.230621  &     5.230621   &    5.764$\times 10^{-13}$ & Type 1\\
		5.230621  &     5.230621   &    1.424$\times 10^{-12}$ & Type 1\\
		5.692104  &     5.692104   &     1.917$\times 10^{-13}$ & Type 1\\
		6.403495  &     6.403495   &    6.314$\times 10^{-13}$ & Type 1\\
		6.734557  &     6.734557   &     4.218$\times 10^{-12}$ & Type 1\\
		6.891325  &     2.839$\times 10^{-11}$  &     6.891325 & Type 2\\
		6.891325  &     9.671$\times 10^{-11}$  &    6.891325 & Type 2\\
		7.760096  &     6.702$\times 10^{-11}$  &     7.760096 & Type 2\\
		\hline		
	\end{tabular}
	
\end{table}
\begin{table}[ht]  %\label{ }
\caption{ The iteration counts in computing the first 20 eigenvalues 
	of the auxiliary Maxwell eigenproblem using different preconditioners
	and different mesh sizes on Domain 1.}
\label{Maxwell_eig_domain_1}
	\centering  
	\begin{tabular}{|c| c| c |c| c |c|c|}
		\hline	
		\multirow{2}*{Level} & \multirow{2}*{$ h $} & \multirow{2}*{d.o.f.} &
		\multicolumn{4}{|c|}{ LOBPCG with preconditioners  } \\
		\cline{4-7}
		~ & ~ & ~ & --- & ILU(0) & SAIT(0.05,10) & MGVC(l,5) \\
		\hline
		$ l =  1 $ & $ \pi/4 $ & $ 300 $ &
		$ 90 $  & $ 19 $  & $ 23 $ & $ 18 $ \\
		\hline
		$ l =  2 $ & $ \pi/8 $ & $ 1,944 $ & 	
		$ 131 $  & $ 17 $  & $ 29 $ & $ 17 $ \\
		\hline
		$ l =  3 $ & $ \pi/16 $ & $ 13,872 $ & 
		$ 182 $  & $ 26 $  & $ 54 $ & $ 16 $ \\
		\hline
		$ l =  4 $ & $ \pi/32 $ & $ 104,544 $ &
		$ 281 $  & $ 56 $  & $ 107 $ & $ 15 $ \\
		\hline
		$ l =  5 $ & $ \pi/64 $ & $ 811,200 $ & 
		$ 642 $  & $ 109 $  & $ 211 $ & $ 14 $ \\
		\hline
	\end{tabular}
	
\end{table}
\begin{table}[ht]  %\label{ }	
	\caption{ The iteration counts in computing the first 20 eigenvalues 
		of the auxiliary Maxwell eigenproblem using different preconditioners
		and different mesh sizes on Domain 2.}
	\label{Maxwell_eig_domain_2}
	\centering  
	\begin{tabular}{|c| c| c |c| c |c|c|}
		\hline	
		\multirow{2}*{Level} & \multirow{2}*{$ h $} & \multirow{2}*{d.o.f.} &
		\multicolumn{4}{|c|}{ LOBPCG with preconditioners  } \\
		\cline{4-7}
		~ & ~ & ~ & --- & ILU(0) & SAIT(0.05,10) & MGVC(l,5) \\
		\hline
		$ l =  1 $ & $ \pi/4 $ & $ 276 $ &
		$ 83 $  & $ 22 $  & $ 25 $ & $ 21 $ \\
		\hline
		$ l =  2 $ & $ \pi/8 $ & $ 1,656 $ & 	
		$ 135 $  & $ 18 $  & $ 30 $ & $ 18 $ \\
		\hline
		$ l =  3 $ & $ \pi/16 $ & $ 11,184 $ & 
		$ 161 $  & $ 24 $  & $ 45 $ & $ 15 $ \\
		\hline
		$ l =  4 $ & $ \pi/32 $ & $ 81,504 $ &
		$ 257 $  & $ 61 $  & $ 96 $ & $ 14 $ \\
		\hline
		$ l =  5 $ & $ \pi/64 $ & $ 620,736 $ & 
		$ 484 $  & $ 111 $  & $ 215 $ & $ 14 $ \\
		\hline
	\end{tabular}
	
\end{table}

\subsection{The grad-div problems}

The discrete weak formulation of the grad-div source problem \eqref{graddiv_source} is:
given $ \vf\in \vL^2(\Omega) $, find $ \vu_h \in \vF_h $
such that
\begin{equation}\label{weak_discrete_graddiv}
	\begin{split}
		\left\langle \nabla \cdot \vu_h \, ,\,  \nabla \cdot \vv_h \right\rangle 
		+  \left\langle \vu_h \, ,\,  \vv_h \right\rangle 
		= \left\langle \vf \, ,\,  \vv_h \right\rangle \qq \vv_h \in \vF_h.
	\end{split}
\end{equation}
Its auxiliary weak form is:
given $ \vf \in \vL^2(\Omega) $, 
find  $ (\vsigma_h, \tilde\vu_h)\in \vE_h \times \vF_h $
such that
\begin{equation}\label{weak_discrete_graddiv_mixed}
	\begin{split}
		\left\langle \vsigma_h \, ,\,  \vtau_h \right\rangle 
		- \left\langle \tilde\vu_h \, ,\,  \nabla \times \vtau_h \right\rangle 
		&=0  \qqq  \qqq\, \vtau_h \in \vE_h,\\
		\left\langle \nabla \times \vsigma_h \, ,\,  \vv_h \right\rangle 
		+ \left\langle \nabla \cdot \tilde \vu_h \, ,\,  \nabla \cdot \vv_h \right\rangle 
		+  \left\langle \tilde\vu_h \, ,\,  \vv_h \right\rangle 
		&=  \left\langle \vf \, ,\,  \vv_h \right\rangle \;\,\qq \vv_h \in \vF_h.
	\end{split}
\end{equation}
The discrete weak form the grad-div eigenvalue problem \eqref{graddiv_eigenvalue} is:
find $ (\lambda_h,\vu_h) \in \mbR \times \vF_h $ such that
\begin{equation}\label{weak_discrete_graddiv_eig}
	\begin{split}
		\left\langle \nabla \cdot \vu_h \, ,\,  \nabla \cdot \vv_h \right\rangle 
		+ c \left\langle \vu_h \, ,\,  \vv_h \right\rangle 
		= \lambda_h \left\langle \vu_h \, ,\,  \vv_h \right\rangle \qq \vv_h \in \vF_h.
	\end{split}
\end{equation}
Its auxiliary problem is:
find $ (\lambda_h,\vu_h) \in \mbR \times \vF_h $ such that
\begin{equation}\label{weak_discrete_graddiv_eig_mixed}
	\begin{split}
		\left\langle \vsigma_h \, ,\,  \vtau_h \right\rangle 
		- \left\langle \tilde\vu_h \, ,\,  \nabla \times \vtau_h \right\rangle 
		&=0  \qqq  \qqq\qqq \, \vtau_h \in \vE_h,\\
		\left\langle \nabla \times \vsigma_h \, ,\,  \vv_h \right\rangle 
		+ \left\langle \nabla \cdot \tilde \vu_h \, ,\,  \nabla \cdot \vv_h \right\rangle  
		&=  \lambda_h \left\langle \vu_h \, ,\,  \vv_h \right\rangle \;\,\qq \vv_h \in \vF_h.
	\end{split}
\end{equation}

Most of the results in computing the grad-div problems using auxiliary schemes
are similar to those of Maxwell problems.
Figure \ref{SAIT_graddiv} is the nonzero entries in SAIT matrices on different meshes and domains.
Figure \ref{graddiv_L5} shows the comparisons of the convergences of
the original grad-div source problem and its auxiliary scheme.
The difference from Maxwell problems is that 
ILU(0) works for the original grad-div source problem,
but it doesn't improve the convergence.
Table \ref{Table_graddiv_Domain_1} and \ref{Table_graddiv_Domain_2} 
show the convergence results on the two domains with different levels of meshes, respectively.
\begin{figure}
	\centering
	\begin{minipage}{6cm}
		\includegraphics[height=5cm,width=6.5cm]{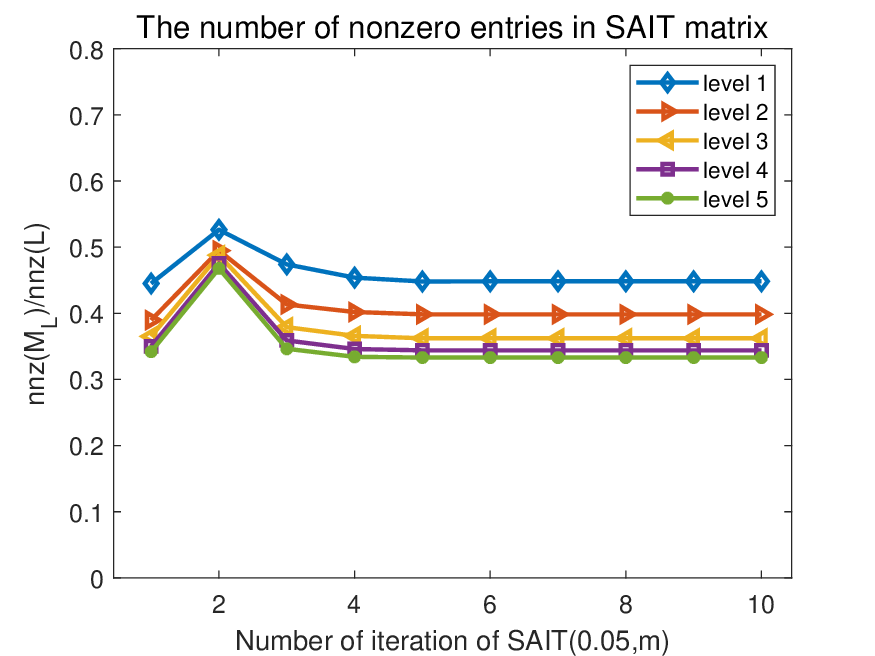}
	\end{minipage}
	\begin{minipage}{6cm}
		\includegraphics[height=5cm,width=6.5cm]{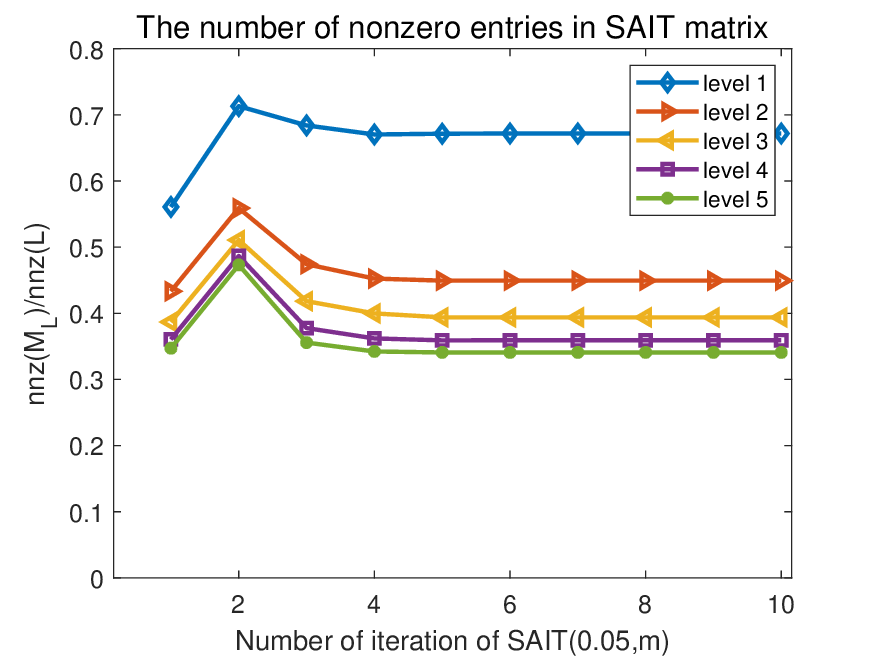}
	\end{minipage}
	\caption{The number of the nonzero entries in the SAIT for the auxiliary grad-div equation.
		Left: Domain 1. Right: Domain 2. }
	\label{SAIT_graddiv}
\end{figure}
\begin{figure}
	\centering
	\begin{minipage}{6cm}
		\includegraphics[height=5.5cm,width=6.5cm]{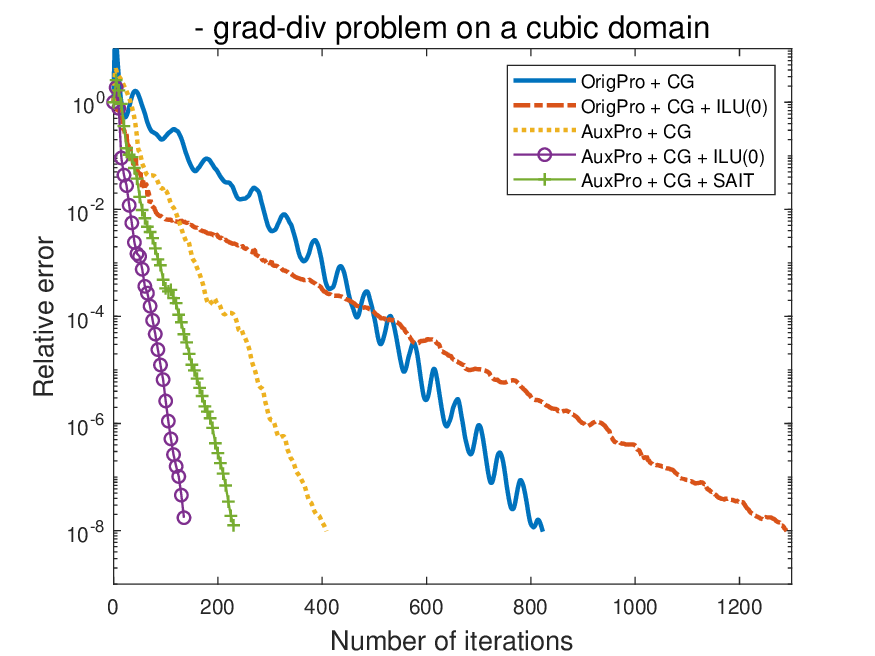}
	\end{minipage}
	\begin{minipage}{6cm}
		\includegraphics[height=5.5cm,width=6.5cm]{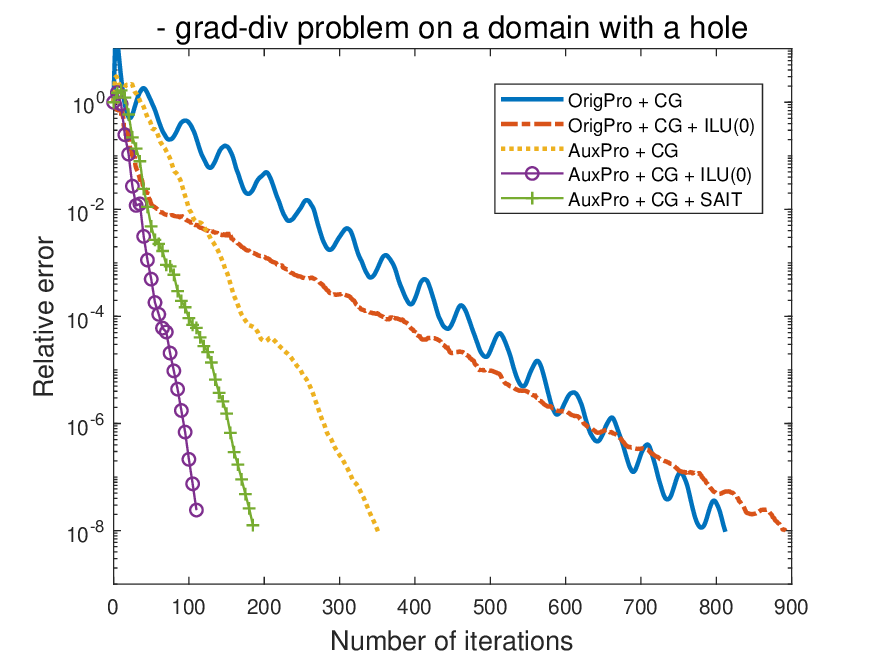}
	\end{minipage}
	\caption{The convergence histories of the original grad-div equation 
		and its auxiliary scheme with different iterative methods. 
		The mesh is level 5 with the mesh size $ h = \pi/64 $.
		Left: Domain 1. Rigth: Domain 2.}
	\label{graddiv_L5}
\end{figure}
\begin{table}[ht]  %\label{ }
	\caption{ The iteration counts of the original discrete grad-div equation,
		its auxiliary equation and the mass equation with different iterative methods
		and preconditioners on Domain 1.}
	\label{Table_graddiv_Domain_1}
	\footnotesize	
	\centering  
	\begin{tabular}{|c|c|c||c|c||c|c|c|c|c||c|c|}	
		\hline
		\multicolumn{3}{|c||}{ } &
		\multicolumn{2}{|c||}{ $ (\mcA + c\mcM_k)u = f $ } &
		\multicolumn{5}{|c||}{ $ (\mcA + \mcB \mcU \mcB^T + c\mcM_k) \tilde  u = f $ } &
		\multicolumn{2}{|c|}{ $ \mcM_k v =  \mcB \mcU \mcB^T \tilde u $ } \\
		\hline
		\multirow{2}*{Level} & \multirow{2}*{$ h $} & \multirow{2}*{d.o.f.} &
		\multicolumn{2}{|c||}{\tabincell{c}{Conjugate \\ Gradient}  } &
		\multicolumn{4}{|c|}{ Conjugate  Gradient } & \multirow{2}*{ \tabincell{c}{Multigrid \\ \tiny{MGVC$(l,5)$}} } &
		\multicolumn{2}{|c|}{ \tabincell{c}{Conjugate \\ Gradient}  }\\
		\cline{4-9} \cline{11-12}
		~ & ~ & ~ & $ -- $ & ILU(0) &
		$ -- $ & ILU(0) & SAIT & \tiny{MGVC$(l,5)$} & ~ & $ -- $ & ILU(0) \\
		\hline
		$ l =  1 $ & $ \pi/4 $ & $ 240 $ & $ 36 $  & $ 19 $  &
		$ 37 $ & $ 10 $ & $ 15 $ &   &   &
		$ 5 $ &  $ 1 $ \\
		\hline
		$ l =  2 $ & $ \pi/8 $ & $ 1,728 $ & $ 99 $  & $ 45 $  &
		$ 64 $ & $ 49 $ & $ 29 $ & $ 4 $ & $ 5 $ &
		$ 9 $ &  $ 1 $ \\
		\hline
		$ l =  3 $ & $ \pi/16 $ & $ 13,056 $ & $ 198 $  & $ 133 $  &
		$ 112 $ & $ 36 $ & $ 57 $ & $ 6 $ & $ 7 $ &
		$ 9 $ &  $ 1 $ \\
		\hline
		$ l =  4 $ & $ \pi/32 $ & $ 101,376 $ & $ 406 $  & $ 398 $  &
		$ 213 $ & $ 70 $ & $ 118 $ & $ 9 $ & $ 10 $ &
		$ 15 $ &  $ 1 $ \\
		\hline
		$ l =  5 $ & $ \pi/64 $ & $ 798,720 $ & $ 820 $  & $ 1297 $  &
		$ 407 $ & $ 138 $ & $ 233 $ & $ 11 $ & $ 15 $ &
		$ 15 $ &  $ 1 $ \\
		\hline
	\end{tabular}
\end{table}
\begin{table}[ht]  %\label{ }
	\caption{ The iteration counts of the original discrete grad-div equation,
	its auxiliary equation and the mass equation with different iterative methods
	and preconditioners on Domain 2.}
\label{Table_graddiv_Domain_2}
	\footnotesize	
	\centering  
	\begin{tabular}{|c|c|c||c|c||c|c|c|c|c||c|c|}	
		\hline
		\multicolumn{3}{|c||}{ } &
		\multicolumn{2}{|c||}{ $ (\mcA + c\mcM_k)u = f $ } &
		\multicolumn{5}{|c||}{ $ (\mcA + \mcB \mcU \mcB^T + c\mcM_k) \tilde u = f $ } &
		\multicolumn{2}{|c|}{ $ \mcM_k v =  \mcB \mcU \mcB^T \tilde u $ } \\
		\hline
		\multirow{2}*{Level} & \multirow{2}*{$ h $} & \multirow{2}*{d.o.f.} &
		\multicolumn{2}{|c||}{ \tabincell{c}{Conjugate \\ Gradient}  } &
		\multicolumn{4}{|c|}{ Conjugate Gradient  } & \multirow{2}*{ \tabincell{c}{Multigrid \\ \tiny{MGVC$(l,5)$}}} &
		\multicolumn{2}{|c|}{ \tabincell{c}{Conjugate \\ Gradient}  }\\
		\cline{4-9} \cline{11-12}
		~ & ~ & ~ & $ -- $ & ILU(0) &
		$ -- $ & ILU(0) & SAIT & MGVC & ~ & $ -- $ & ILU(0) \\
		\hline
		$ l =  1 $ & $ \pi/4 $ & $ 204 $ & $ 55 $  & $ 17 $  &
		$ 45 $ & $ 10 $ & $ 14 $ &   &   &
		$ 6 $ &  $ 1 $ \\
		\hline
		$ l =  2 $ & $ \pi/8 $ & $ 1,392 $ & $ 109 $  & $ 36 $  &
		$ 68 $ & $ 17 $ & $ 27 $ & $ 4 $ & $ 6 $ &
		$ 12 $ &  $ 1 $ \\
		\hline
		$ l =  3 $ & $ \pi/16 $ & $ 10,176 $ & $ 211 $  & $ 103 $  &
		$ 111 $ & $ 31 $ & $ 48 $ & $ 6 $ & $ 6 $ &
		$ 17 $ &  $ 1 $ \\
		\hline
		$ l =  4 $ & $ \pi/32 $ & $ 77,568 $ & $ 401 $  & $ 295 $  &
		$ 197 $ & $ 59 $ & $ 94 $ & $ 7 $ & $ 9 $ &
		$ 16 $ &  $ 1 $ \\
		\hline
		$ l =  5 $ & $ \pi/64 $ & $ 605,184 $ & $ 816 $  & $ 884 $  &
		$ 365 $ & $ 113 $ & $ 186 $ & $ 10 $ & $ 13 $ &
		$ 15 $ &  $ 1 $ \\
		\hline
	\end{tabular}
\end{table}

Table \ref{recompute_grad_div_Domain_1} and \ref{recompute_grad_div_Domain_2} show the results in 
recomputing the numerical eigenvalues using
the original discrete eigenvalue problem \eqref{weak_discrete_graddiv_eig}.
Table \ref{grad_div_eig_domain_2} and \ref{grad_div_eig_domain_1} show more convergence results
of auxiliary discrete grad-div eigenvalue problems \eqref{weak_discrete_graddiv_eig_mixed} on 
the two domains, respectively.

\begin{table}[ht]  %\label{ }
	\caption{ Recompute the numerical eigenvalues using 
		the original discrete grad-div operator
		and the auxiliary operator on Domain 1, respectively.}
	\label{recompute_grad_div_Domain_1}	
	\footnotesize
	\centering  
	\begin{tabular}{|c|c|c|c|c| }	
		\hline
		Exact	&  $\left( \mcA + \mcB \mcU  \mcB^T \right)  u = \lambda_h \mcM_k u$
		& $ \mcA u = \lambda_h \mcM_k u$ 
		&  $ \mcB \mcU  \mcB^T u = \lambda_h \mcM_k u$ & Type\\
		\hline	
		3	&	3.000602 &     3.000602   &   7.452 $\times 10^{-13}$ & Type 1\\
		6	&	6.003615 &     6.003615   &   4.122 $\times 10^{-13}$ & Type 1\\
		6	&	6.003615 &     6.003615   &   1.045 $\times 10^{-12}$ & Type 1\\
		6	&	6.003615 &     6.003615   &   1.128 $\times 10^{-12}$ & Type 1\\
		9	&	9.006628 &     9.006628   &   5.124 $\times 10^{-13}$ & Type 1\\
		9	&	9.006628  &    9.006628   &   8.570 $\times 10^{-13}$ & Type 1\\
		9	&	9.006628  &    9.006628   &   1.369 $\times 10^{-12}$ & Type 1\\
		&	9.539861 &     1.242 $\times 10^{-12}$  &   9.539861  & Type 2\\
		&	9.539861 &     1.989 $\times 10^{-12}$  &   9.539861  & Type 2\\
		&	9.539861  &    3.823 $\times 10^{-12}$  &   9.539861  & Type 2\\
		11	&	11.016677 &     11.016677  &    1.086 $\times 10^{-12}$ & Type 1\\
		11	&	11.016677 &     11.016677  &    1.373 $\times 10^{-12}$ & Type 1\\
		11	&	11.016677  &    11.016677  &   3.347 $\times 10^{-12}$ & Type 1\\
		12	&	12.009641  &    12.009641  &    8.625 $\times 10^{-13}$ & Type 1\\
		14	&	14.019690 &     14.019690  &     2.515 $\times 10^{-13}$ & Type 1\\
		14	&	14.019690 &     14.019690  &    9.587 $\times 10^{-13}$ & Type 1\\
		14	&	14.019690  &    14.019690  &     2.729 $\times 10^{-12}$ & Type 1\\
		14	&	14.019690  &    14.019690  &    6.097 $\times 10^{-12}$ & Type 1\\
		14	&	14.019690  &    14.019690  &     4.431 $\times 10^{-12}$ & Type 1\\
		14	&	14.019690  &    14.019690  &    1.004 $\times 10^{-13}$ & Type 1\\
		\hline		
		
	\end{tabular}

\end{table}

\begin{table}[ht]  %\label{ }
\caption{ Recompute the numerical eigenvalues using 
	the original discrete grad-div operator
	and the auxiliary operator on Domain 2, respectively.}
\label{recompute_grad_div_Domain_2}
	\footnotesize	
	\centering  
	\begin{tabular}{|c|c|c|c| }	
		\hline
		$\left( \mcA + \mcB \mcU  \mcB^T \right)  u = \lambda_h \mcM_k u$
		& $ \mcA u = \lambda_h \mcM_k u$ 
		&  $ \mcB \mcU  \mcB^T u = \lambda_h \mcM_k u$ & Type\\
		\hline	
		4.533215  &   4.856 $\times 10^{-12}$  &   4.533215  & Type 2\\
		6.952396  &   1.882 $\times 10^{-12}$  &   6.952396  & Type 2\\
		6.952396  &   2.652 $\times 10^{-12}$  &   6.952396  & Type 2\\
		12.455297  &   3.875 $\times 10^{-13}$  &   12.455297  & Type 2\\
		15.488922  &   15.488922   &   4.132 $\times 10^{-13}$ & Type 1\\
		15.805105  &   5.086 $\times 10^{-13}$  &   15.805105  & Type 2\\
		15.921176  &   15.921176   &   5.413 $\times 10^{-13}$ & Type 1\\
		15.921176  &   15.921176   &   4.822 $\times 10^{-13}$ & Type 1\\
		16.618025  &   16.618025   &   1.017 $\times 10^{-13}$ & Type 1\\
		18.109571  &   7.399 $\times 10^{-13}$  &   18.109571  & Type 2\\
		18.491935  &   18.491935   &   4.263 $\times 10^{-13}$ & Type 1\\
		18.924189  &   18.924189   &   3.922 $\times 10^{-13}$ & Type 1\\
		18.924189  &   18.924189   &   6.295 $\times 10^{-13}$ & Type 1\\
		19.251186  &   19.251186   &   2.779 $\times 10^{-13}$ & Type 1\\
		19.621038  &   19.621038   &   1.210 $\times 10^{-12}$ & Type 1\\
		20.546486  &   1.087 $\times 10^{-12}$  &   20.546486  & Type 2\\
		20.546486  &   2.856 $\times 10^{-12}$  &   20.546486  & Type 2\\
		20.932342  &   20.932342   &   1.935 $\times 10^{-12}$ & Type 1\\
		20.932342  &   20.932342   &   3.757 $\times 10^{-12}$ & Type 1\\
		22.254199  &   22.254199   &   1.524 $\times 10^{-09}$ & Type 1\\
		\hline		
	\end{tabular}
	
\end{table}

\begin{table}[ht]  %\label{ }
\caption{ The iteration counts in computing the first 20 eigenvalues 
	of the auxiliary grad-div eigenproblem using different preconditioners
	and different mesh sizes on Domain 1.}
\label{grad_div_eig_domain_1}
	\centering  
	\begin{tabular}{|c| c| c |c| c |c|c|}
		\hline	
		\multirow{2}*{Level} & \multirow{2}*{$ h $} & \multirow{2}*{d.o.f.} &
		\multicolumn{4}{|c|}{ LOBPCG with preconditioners  } \\
		\cline{4-7}
		~ & ~ & ~ & $ -- $ & ILU(0) & SAIT(0.05,10) & MGVC(l,5) \\
		\hline
		$ l =  1 $ & $ \pi/4 $ & $ 240 $ &
		$ 83 $  & $ 26 $  & $ 30 $ & $ 25 $ \\
		\hline
		$ l =  2 $ & $ \pi/8 $ & $ 1,728 $ & 	
		$ 167 $  & $ 39 $  & $ 57 $ & $ 34 $ \\
		\hline
		$ l =  3 $ & $ \pi/16 $ & $ 13,056 $ & 
		$ 175 $  & $ 42 $  & $ 67 $ & $ 18 $ \\
		\hline
		$ l =  4 $ & $ \pi/32 $ & $ 101,376 $ &
		$ 303 $  & $ 78 $  & $ 132 $ & $ 17 $ \\
		\hline
		$ l =  5 $ & $ \pi/64 $ & $ 798,720 $ & 
		$ 655 $  & $ 167 $  & $ 305 $ & $ 17 $ \\
		\hline
	\end{tabular}
	
\end{table}

\begin{table}[ht]  %\label{ }
\caption{ The iteration counts in computing the first 20 eigenvalues 
	of the auxiliary grad-div eigenproblem using different preconditioners
	and different mesh sizes on Domain 2.}
\label{grad_div_eig_domain_2}
	\centering  
	\begin{tabular}{|c| c| c |c| c |c|c|}
		\hline	
		\multirow{2}*{Level} & \multirow{2}*{$ h $} & \multirow{2}*{d.o.f.} &
		\multicolumn{4}{|c|}{ LOBPCG with preconditioners  } \\
		\cline{4-7}
		~ & ~ & ~ & --- & ILU(0) & SAIT(0.05,10) & MGVC(l,5) \\
		\hline
		$ l =  1 $ & $ \pi/4 $ & $ 204 $ &
		$ 75 $  & $ 23 $  & $ 25 $ & $ 22 $ \\
		\hline
		$ l =  2 $ & $ \pi/8 $ & $ 1,392 $ & 	
		$ 168 $  & $ 40 $  & $ 52 $ & $ 37 $ \\
		\hline
		$ l =  3 $ & $ \pi/16 $ & $ 10,176 $ & 
		$ 179 $  & $ 40 $  & $ 65 $ & $ 22 $ \\
		\hline
		$ l =  4 $ & $ \pi/32 $ & $ 77,568 $ &
		$ 451 $  & $ 92 $  & $ 141 $ & $ 23 $ \\
		\hline
		$ l =  5 $ & $ \pi/64 $ & $ 605,184 $ & 
		$ 739 $  & $ 187 $  & $ 269 $ & $ 22 $ \\
		\hline
	\end{tabular}
	
\end{table}

\begin{remark}
	When solving the auxiliary eigenvalue problems,
	the iteration counts of the multigrid preconditioners
	tend to a stable number with the mesh becoming finer,
	shown in Table \ref{Maxwell_eig_domain_1},
	\ref{Maxwell_eig_domain_2},
	\ref{grad_div_eig_domain_1} 
	and \ref{grad_div_eig_domain_2}.
	However, when solving the linear systems, 
	shown in Table \ref{Table_Maxwell_Domain_1}, 
	\ref{Table_Maxwell_Domain_2}, 
	\ref{Table_graddiv_Domain_1} 
	and \ref{Table_graddiv_Domain_2},
	the iteration counts of the multigrid methods
	increase with the mesh becoming finer.
	There are the similar results in solving 
	the vector Laplacian \cite{mixed_vector}.
	There may be other reasons on the increasing multigrid iterations
	in this paper.
	For example,
	the inverses of the mass matrix $ \mcM_{k-1} $
	in the matrix form of discrete Hodge Laplacian
	are replaced by diagonal matrices,
	which means that the auxiliary problems are not
	exact discrete Hodge Laplacian problems.
	Or there are some bugs contained in the author's codes.
\end{remark}

\clearpage
\section{Conclusions}\label{sec_conclusion}

We construct auxiliary iterative schemes 
for the operator $ d^*d $ using the Hodge Laplacian.
The auxiliary schemes are easy to implement
using the finite element spaces 
on the corresponding finite element complex.
If the complex is Fredholm,
the auxiliary part complements the huge kernel 
contained in the discrete operator of $ d^*d $.
Then the distributions of the spectra of 
the auxiliary problems become Laplace-like.
To make the systems sparse,
we prove that the inverse of a mass matrix
in the system of the discrete Hodge Laplacian
can be replaced by a sparse SPD matrix.
In actual computations,
it is enough to compute the sparse Hodge Laplacian problems.
Many iterative methods and preconditioning techniques
that are efficient for Laplace problems
are also efficient for the auxiliary problems
with some simple modifications.
The auxiliary schemes of the source and eigenvalue problems 
can be computed almost in the same way as Laplace problems.
After obtaining the solutions of the auxiliary problems,
the desired solutions can be recovered or identified easily.

We present the Maxwell and grad-div operators
as numerical examples to
verify the performance of the auxiliary schemes.
We use the ILU method and geometric multigrid method 
as preconditioning to solve these discrete problems.
The numerical tests are all three-dimensional
and include various cases,
source problems and eigenvalue problems,
convex and non-convex domains,
different iterative methods and preconditioners.
The convergence tendencies of the auxiliary problems
are similar to Laplace problems.
The results show the efficiency of the auxiliary schemes.

In this paper,
we mainly prove and verify that the solutions of the  $ d^*d $ problems
can be obtained through the auxiliary Hodge Laplacian problems.
The discrete Hodge Laplacian problems play the central role 
in these auxiliary schemes.
We do not do much analyses on the linear solvers and eigensolvers
for the discrete Hodge Laplacian.
The theories, algorithms and implementations in solving 
the discrete Hodge Laplacian problems
need to be studied further.

\

\

\textbf{Acknowledgements.}
The author would like to thank Prof. Yan Xu
for the funding support at 
University of Science and Technology of China
and 
thank Prof. Weiying Zheng and Prof. Shuo Zhang
at State Key Laboratory of Scientific and Engineering Computing, 
Chinese Academy of Sciences
for useful discussions

	%\begin{acknowledgements}
	%\end{acknowledgements}


\bibliographystyle{plain} 
\bibliography{reference_aux} 	
	
\end{document}